\documentclass[11pt]{amsart}
\usepackage{amssymb}
\usepackage{amsmath}
\usepackage{amsthm}
\usepackage{bbm}
\usepackage{dsfont}
\newtheorem{thm}{Theorem}
\newtheorem{prop}{Proposition}

\newtheorem{lem}{Lemma}
\newtheorem{exam}{Example}
\newtheorem{cor}{Corollary}
\newtheorem{rem}{Remark}

\newtheorem{definition}{Definition}
\newtheorem{question}{Question}
\numberwithin{equation}{section}

\DeclareMathOperator{\rp}{\xrightarrow[]{rp}}
\DeclareMathOperator{\nc}{\xrightarrow[]{\lVert \cdot \rVert_X}}

\DeclareMathOperator{\wc}{\xrightarrow[]{w}}
\DeclareMathOperator{\unc}{\xrightarrow[]{un}}
\DeclareMathOperator{\uoc}{\xrightarrow[]{uo}}

\DeclareMathOperator{\upc}{\xrightarrow[]{up}}

\DeclareMathOperator{\oc}{\xrightarrow[]{o}}

\DeclareMathOperator{\pc}{\xrightarrow[]{p}}
\DeclareMathOperator{\norm}{\xrightarrow[]{\| \cdot \|}}

\begin{document}
	
\title{AMS Journal Sample}
	
\author{A. Ayd{\i}n$^{1,4}$, E. Yu. Emelyanov$^{1,2}$, N. Erkur\c{s}un \"Ozcan$^3$, M. A. A. Marabeh$^1$}

\address{$^{1}$ Department of Mathematics, Middle East Technical University, Ankara, 06800 Turkey.} 

\email{{aaydin.aabdullah@gmail.com} and {eduard@metu.edu.tr}}
\email{{mohammad.marabeh@metu.edu.tr} and {m.maraabeh@gmail.com}}

\address{$^{2}$ Sobolev Institute of Mathematics, Novosibirsk, 630090, Russia.}
\email{{emelanov@math.nsc.ru}}

\address{$^{3}$ Department of Mathematics, Hacettepe University, Ankara, 06800, Turkey.}
\email{{erkursun.ozcan@hacettepe.edu.tr}}

\address{$^{4}$ Department of Mathematics, Mu\c{s} Alparslan University, Mu\c{s}, 49250, Turkey.}
\email{{a.aydin@alparslan.edu.tr}}

\date{20.01.2017}
	
	
\title{Compact-Like Operators in Lattice-Normed Spaces}
	
\begin{abstract}
A linear operator $T$ between two lattice-normed spaces is said to be $p$-compact if, for any $p$-bounded net $x_\alpha$, 
the net $Tx_\alpha$ has a $p$-convergent subnet. $p$-Compact operators generalize several known classes of operators such as compact, 
weakly compact, order weakly compact, $AM$-compact operators, etc. Similar to $M$-weakly and $L$-weakly compact operators, 
we define $p$-$M$-weakly and $p$-$L$-weakly compact operators and study some of their properties. We also study $up$-continuous 
and $up$-compact operators between lattice-normed vector lattices.
\end{abstract}

\maketitle
	
\section{Introduction}
It is known that order convergence in vector lattices is not topological in general. Nevertheless, via order convergence, continuous-like operators (namely, order continuous operators) 
can be defined in vector lattices without using any topological structure. On the other hand, compact operators play an important role in functional analysis. 
Our aim in this paper is to introduce and study compact-like operators in lattice-normed spaces and in lattice-normed vector lattices by developing topology-free techniques. 

Recall that a net $(x_\alpha)_{\alpha\in A}$ in a vector lattice $X$ is order convergent (or $o$-convergent, for short) to $x\in X$, 
if there exists another net $(y_\beta)_{\beta\in B}$ satisfying $y_\beta \downarrow 0$, and for any $\beta\in B$, there exists $\alpha_\beta\in A$ 
such that $|x_\alpha-x|\leq y_\beta$ for all $\alpha\geq\alpha_\beta$. In this case we write $x_\alpha\oc x$. In a vector lattice $X$, 
a net $x_\alpha$ is unbounded order convergent (or $uo$-convergent, for short) to $x\in X$ if $|x_\alpha-x|\wedge u\oc 0$ for every $u\in X_+$; see \cite{GTX}. 
In this case we write $x_\alpha\uoc x$. In a normed lattice $(X,\left\|\cdot\right\|)$, a net $x_\alpha$ is unbounded norm convergent to $x\in X$, 
written as $x_\alpha\unc x$, if $\left\||x_\alpha-x|\wedge u\right\|\to 0$ for every $u\in X_+$; see \cite{DOT}. 
Clearly, if the norm is order continuous then $uo$-convergence implies $un$-convergence. Throughout the paper, all vector lattices are assumed to be real and Archimedean.

Let $X$ be a vector space, $E$ be a vector lattice, and $p:X\to E_+$ be a vector norm (i.e. $p(x)=0\Leftrightarrow x=0$, 
$p(\lambda x)=|\lambda|p(x)$ for all $\lambda\in\mathbb{R}$, $x\in X$, and $p(x+y)\leq p(x)+p(y)$ for all $x,y\in X$) then the triple $(X,p,E)$ 
is called a {\em lattice-normed space}, abbreviated as LNS. The lattice norm $p$ in an LNS $(X,p,E)$ is said to be {\em decomposable} if for all $x\in X$ and
$e_1,e_2\in E_+$, it follows from $p(x)=e_1+e_2$, that there exist $x_1,x_2\in X$ such that $x=x_1+x_2$ and $p(x_k)=e_k$ for $k=1,2$. If $X$ is a vector lattice, 
and the vector norm $p$ is monotone (i.e. $|x|\leq |y|\Rightarrow p(x)\leq p(y)$) then the triple $(X,p,E)$ is called a {\em lattice-normed vector lattice}, abbreviated as LNVL. In this article we usually use the pair $(X,E)$ or just $X$ to refer to an LNS  $(X,p,E)$ if there is no confusion.

We abbreviate the convergence $p(x_{\alpha}-x)\oc 0$ as $x_\alpha\pc x$ and say in this case that $x_\alpha$ $p$-converges to $x$. 
A net $(x_\alpha)_{\alpha \in A}$ in an LNS $(X,p,E)$ is said to be {\em $p$-Cauchy} if the net $(x_\alpha-x_{\alpha'})_{(\alpha,\alpha') \in A\times A}$ $p$-converges to $0$. 
An LNS $(X,p,E)$ is called (\textit{sequentially}) {\em $p$-complete} if every $p$-Cauchy (sequence) net in $X$ is $p$-convergent. In an LNS $(X,p,E)$ a subset $A$ of $X$ is called {\em $p$-bounded} if there exists $e\in E$ such that $p(a)\leq e$ for all $a\in A$. An LNVL $(X,p,E)$ is called {\em $op$-continuous} if $x_\alpha\oc 0$ implies that $p(x_\alpha)\oc 0$. 

A net $x_\alpha$ in an LNVL $(X,p,E)$ is said to be {\em unbounded $p$-convergent} to $x\in X$ (shortly, $x_\alpha$ $up$-converges to $x$ or  
$x_\alpha\upc x$), if $p(|x_\alpha-x|\wedge u)\oc 0$ for all $u\in X_+;$ see \cite[Def.6]{AEEM}.

Let $(X,p,E)$ be an LNS and $(E,\lVert\cdot\rVert_E)$ be a normed lattice. The \textit{mixed norm} on $X$ is defined by $p\text{-}\lVert x\rVert_E=\lVert p(x)\rVert_E$ for all $x\in X$. In this case the normed space $(X,p\text{-}\lVert\cdot\rVert_E)$ is called a \textit{mixed-normed space} (see, for example \cite[7.1.1, p.292]{K}).

A net $x_\alpha$ in an LNS $(X,p,E)$ is said to {\em relatively uniformly $p$-converge} to $x\in X$ (written as, $x_\alpha\rp x$) if there is $e\in E_+$ such that for any $\varepsilon>0,$ there is $\alpha_\varepsilon$ satisfying $p(x_\alpha-x)\leq \varepsilon e$ for all $\alpha\geq \alpha_\varepsilon.$ In this case we say that $x_\alpha$ $rp$-converges to $x$. A net $x_\alpha$ in an LNS $(X,p,E)$ is called \textit{$rp$-Cauchy} if the net $(x_\alpha-x_{\alpha'})_{(\alpha,\alpha') \in A\times A}$ $rp$-converges to $0.$
It is easy to see that for a sequence $x_n$ in an LNS $(X,p,E)$, $x_n\rp x$ iff there exist $e\in E_+$ and a numerical sequence $\varepsilon_k\downarrow 0$ such that for all $k\in  \mathbb{N}$ and there is $n_k\in  \mathbb{N}$ satisfying $p(x_n-x)\leq \varepsilon_ke$ for all $n\geq n_k.$ An LNS $(X,p,E)$ is said to be \textit{$rp$-complete} if every $rp$-Cauchy sequence in $X$ is $rp$-convergent. It should be noticed that in a $rp$-complete LNS every $rp$-Cauchy net is $rp$-convergent. Indeed, assume $x_\alpha$ is a $rp$-Cauchy net in a $rp$-complete LNS $(X,p,E)$. Then an element $e\in E_+$ exists such that, for all $n\in \mathbb{N},$ there is an $\alpha_n$ such that $p(x_{\alpha'}-x_\alpha)\leq \frac{1}{n}e$ for all $\alpha,\alpha'\geq\alpha_n$. We select a strictly increasing sequence $\alpha_n$. Then it is clear that $x_{\alpha_n}$ is $rp$-Cauchy sequence, and so there is $x\in X$ such that $x_{\alpha_n}\rp x$. Let $n_0\in\mathbb{N}$. Hence, there is $\alpha_{n_0}$ such that for all $\alpha\geq \alpha_{n_0}$ we have $p(x_{\alpha}-x_{\alpha_{n_0}}) \leq \frac{1}{n_0}e$ and, for all $n \geq n_0$ $p(x-x_{\alpha_{n_0}}) \leq \frac{1}{n_0}e$, from which it follows that $x_{\alpha}\rp x$.

We recall the following result (see for example \cite[7.1.2,p.293]{K}). If $(X,p,E)$ is an LNS such that $(E,\lVert\cdot\rVert_E)$ is a Banach space 
then $(X,p\text{-}\lVert\cdot\rVert_E)$ is norm complete iff the LNS $(X,p,E)$ is $rp$-complete. 
On the other hand, it is not difficult to see that if an LNS is sequentially $p$-complete then it is $rp$-complete. Thus, the following result follows readily.

\begin{lem}\label{mixednormisbanach}
Let $(X,p,E)$ be an LNS such that $(E,\lVert\cdot\rVert_E)$ is a Banach space. If $(X,p,E)$ is sequentially $p$-complete then $(X,p\text{-}\lVert\cdot\rVert_E)$ is a Banach space.
\end{lem}

Consider LNSs $(X,p,E)$ and $(Y,m,F)$. A linear operator $T:X\to Y$ is said to be \textit{dominated} if there is a positive operator $S:E\to F$ satisfying $m(Tx)\leq S(p(x))$ for all $x\in X$. In this case, 
$S$ is called a \textit{dominant} for $T$. The set of all dominated operators from $X$ to $Y$ is denoted by $M(X,Y)$. In the ordered vector space $L^\sim(E,F)$ of all order bounded operators from $E$ into $F$, 
if there is a least element of all dominants of an operator $T$ then such element is called the \textit{exact dominant} of $T$ and denoted by $\pmb{\lvert}T\pmb{\rvert}$; see \cite[4.1.1,p.142]{K}. 

By considering \cite[4.1.3(2),p.143]{K} and Kaplan's example \cite[Ex.1.17]{AB}, we see that not every dominated operator possesses an exact dominant. 
On the other hand if $X$ is decomposable and  $F$ is order complete then every dominated operator $T:X\to Y$ has an exact dominant $\pmb{\lvert}T\pmb{\rvert}$; see \cite[4.1.2,p.142]{K}. 

We refer the reader for more information on LNSs to \cite{BGKKKM,E,KK,K} and \cite{AEEM}. It should be noticed that the theory of lattice-normed spaces is well- developed 
in the case of decomposable lattice norms (cf. \cite{KK,K}). In \cite{C} and \cite{P} the authors studied some classes of operators in LNSs under the assumption that the lattice norms are decomposable.  
In this article, we usually do not assume lattice norms to be decomposable. 

Throughout this article, $L(X,Y)$ denotes the space of all linear operators between  vector spaces $X$ and $Y$. For normed spaces $X$ and $Y$ we use $B(X,Y)$ for the space of all norm bounded linear operators from $X$ into $Y$. We write $L(X)$ for $L(X,X)$ and for $B(X)$ for $B(X,X)$. If $X$ is a normed space then $X^*$ denotes the topological dual of $X$ and $B_X$ denotes the closed unit ball of $X.$ For any set $A$ of a vector lattice $X$, we denote by $sol(A)$ the solid hull of $A$, i.e. 
$sol(A)=\{x\in X:\lvert x\rvert\leq\lvert a\rvert \ \text{for some}\ a\in A\}$.

The following standard fact will be used throughout this article.
\begin{lem}\label{subseq criterion}
Let $(X,\lVert\cdot\rVert)$ be a normed space. Then $x_n\norm x$ iff for any subsequence $x_{n_k}$ there is a further subsequence $x_{n_{k_j}}$ such that $x_{n_{k_j}}\norm x$.
\end{lem}

The structure of this paper is as follows. In section 2, we recall definitions of $p$-continuous and $p$-bounded operators between LNSs. We study the relation between $p$-continuous operators 
and norm continuous operators acting in mixed-normed spaces; see Proposition \ref{seqpcontimpliesnormcont} and Theorem \ref{normcontimpliespcont}. We show that every $p$-continuous operator is $p$-bounded. 
We end this section by giving a generalization of the fact that {\em any positive operator from a Banach lattice into a normed lattice is norm bounded} in Theorem \ref{positivityandboundedness}.

In section 3, we introduce the notions of $p$-compact and sequentially $p$-compact operator between LNSs. These operators generalize several known classes of operators such as compact, weakly compact, 
order weakly compact, and $AM$-compact operators; see Example \ref{examplesofp-compact}. Also the relation between sequentially $p$-compact operators and compact operators acting in mixed-normed spaces are investigated; 
see Propositions \ref{compactimpliesseq.p-compact} and \ref{seq.p-compactimpliescompact}. Finally we introduce the notion of a $p$-semicompact operator and study some of its properties.

In section 4, we define $p$-$M$-weakly and $p$-$L$-weakly compact operators which correspond respectively to $M$-weakly and $L$-weakly compact operators. Several properties of these operators are investigated.

In section 5, the notions of (sequentially) $up$-continuous and (sequentially) $up$-compact operators acting between LNVLs, are introduced. Composition of a sequentially $up$-compact operator 
with a dominated lattice homomorphism is considered in Theorem \ref{sequentially $up$-compactness}, Corollary \ref{rangeisup-regular}, and Corollary \ref{idealisup-regular}.

\section{$p$-Continuous and $p$-Bounded Operators}
In this section we recall the notion of a $p$-continuous operator in an LNS which generalizes the notion of order continuous operator in a vector lattice. 

\begin{definition}
Let $X$, $Y$ be two LNSs and $T\in L(X,Y)$. Then
\begin{enumerate}
\item[(1)] $T$ is called {\em $p$-continuous} if $x_\alpha\pc 0$ in $X$  implies $Tx_\alpha\pc 0$ in $Y$. 
If the condition holds only for sequences then $T$ is called {\em sequentially $p$-continuous}.
\item[(2)] $T$ is called {\em $p$-bounded} if it maps $p$-bounded sets in $X$ to $p$-bounded sets in $Y$.
\end{enumerate}
\end{definition}
\noindent\

\begin{rem}\ \
\begin{enumerate}
\item[(i)] The collection of all $p$-continuous operators between LNSs is a vector space.
\item[(ii)] Using $rp$-convergence one can introduce the following notion:\\
 A linear operator $T$ from an LNS $(X,E)$ into another LNS $(Y,F)$ is called \textit{$rp$-continuous} if $x_\alpha\rp 0$ in $X$  implies $Tx_\alpha\rp 0$ in $Y$. But this notion is not that interesting because it coincides with $p$-boundedness of an operator $($see \cite[Thm. 5.3.3 (a) ]{BGKKKM}$).$     
\item[(iii)] A $p$-continuous (respectively, sequentially $p$-continuous ) operator between two LNSs is also known as $bo$-continuous (respectively, sequentially $bo$-continuous)  see e.g. \cite[4.3.1,p.156]{K}.

\item[(iv)] Let $(X,E)$ be a decomposable LNS and let $F$ be an order complete vector lattice. 
Then $T\in M_n(X,Y)$ iff its exact dominant $\pmb{\lvert} T \pmb{\rvert}$ is order continuous \cite[Thm.4.3.2]{K}, where $M_n(X,Y)$ denotes the set of all dominated $bo$-continuous operators from $X$ to $Y$.
\item[(v)] Every dominated operator is $p$-bounded. The converse not need be true, for example consider 
the identity operator $I:(\ell_\infty,\lvert \cdot \rvert,\ell_\infty)\to(\ell_\infty,\lVert\cdot\rVert,\mathbb{R})$. It is $p$-bounded but not dominated $($see \cite[Rem.,p.388]{BGKKKM}$)$.
\end{enumerate}
\end{rem}

\noindent

Next we illustrate $p$-continuity and $p$-boundedness of operators in particular LNSs.

\begin{exam}\ \
\begin{enumerate} 
\item[(i)] Let $X$ and $Y$ be vector lattices then $T\in L(X,Y)$ is $($$\sigma$-$)$ order continuous iff $T:(X,\lvert\cdot\rvert,X)\to(Y,\lvert\cdot\rvert,Y)$ is $($sequentially$)$ $p$-continuous.
\item[(ii)] Let $X$ and $Y$ be vector lattices then $T\in L^\sim(X,Y)$ iff  $T:(X,\lvert \cdot \rvert,X)\to(Y,\lvert \cdot \rvert,Y)$ is $p$-bounded.
\item[(iii)] Let $(X,\lVert\cdot\rVert_X)$ and $(Y,\lVert\cdot\rVert_Y)$ be normed spaces then $T\in B(X,Y)$ iff $T:(X,\lVert \cdot \rVert_X,\mathbb{R})\to(Y,\lVert \cdot \rVert_Y,\mathbb{R})$ is $p$-continuous iff 
$T:(X,\lVert\cdot\rVert_X,\mathbb{R})\to(Y,\lVert\cdot\rVert_Y,\mathbb{R})$ is $p$-bounded.
\item[(iv)] Let $X$ be a vector lattice and $(Y,\lVert\cdot\rVert_Y)$ be a normed space. Then $T\in L(X,Y)$ is called \textit{order-to-norm continuous} 
if $x_\alpha\oc 0$ in $X$ implies $Tx_\alpha\xrightarrow{\lVert\cdot\rVert_Y} 0$, see \cite[Sect.4,p.468]{MMP}. 
Therefore, $T:X\to Y$ is order-to-norm continuous iff $T:(X,\lvert\cdot\rvert,X)\to(Y,\lVert\cdot\rVert_Y,\mathbb{R})$ is $p$-continuous.
\end{enumerate}
\end{exam}

\begin{lem}
Given an $op$-continuous LNVL $(Y,m,F)$ and a vector lattice $X$. If $T:X\to Y$ is $($$\sigma$-$)$ order continuous then $T:(X,\lvert\cdot\rvert,X)\to(Y,m,F)$ is $($sequentially$)$ $p$-continuous.
\end{lem}

\begin{proof}
Assume that $X\ni x_\alpha\pc 0$ in $(X,\lvert\cdot\rvert,X)$ then $x_\alpha\oc 0$ in $X$. Thus, $Tx_\alpha\oc 0$ in $Y$ as $T$ is order continuous. 
Since $(Y,m,F)$ is $op$-continuous then $m(Tx_\alpha)\oc 0$ in $F$. Therefore, $Tx_\alpha\pc 0$ in $Y$ and so $T$ is $p$-continuous.

The sequential case is similar.
\end{proof}

\begin{prop}\label{Toc}
Let $(X,p,E)$ be an $op$-continuous LNVL, $(Y,m,F)$ be an LNVL and $T:(X,p,E)\to(Y,m,F)$ be a $($sequentially$)$ $p$-continuous positive operator. 
Then $T:X\to Y$ is $($$\sigma$-$)$ order continuous.
\end{prop}

\begin{proof}
We show only the order continuity of $T$, the sequential case is analogous.
Assume $x_\alpha\downarrow 0$ in $X$. Since $X$ is $op$-continuous then $p(x_\alpha)\downarrow 0$. Hence, $x_\alpha\pc 0$ in $X$. By the $p$-continuity of $T$, 
we have $m(Tx_\alpha)\oc 0$ in $F$. Since $0\leq T$ then $Tx_\alpha\downarrow$. Also we have $m(Tx_\alpha)\oc 0$, so it follows from \cite[Prop.1]{AEEM} that $Tx_\alpha\downarrow 0$. 
Thus, $T$ is order continuous.
\end{proof}

\begin{cor}
Let $(X,p,E)$ be an $op$-continuous LNVL, $(Y,m,F)$ be an LNVL such that $Y$ is order complete. If $T:(X,p,E)\to(Y,m,F)$ is $p$-continuous and $T\in L^\sim (X,Y)$ then $T:X\to Y$ is order continuous.
\end{cor}

\begin{proof}
Since  $Y$ is order complete and  $T$ is order bounded then $T=T^+-T^-$ by Riesz-Kantorovich formula. 
Now, Proposition \ref{Toc} implies that $T^+$ and $T^-$ are both order continuous. Hence, $T$ is also order continuous.
\end{proof}

\begin{prop}
Let $(X,\lVert\cdot\rVert_X)$ be a $\sigma$-order continuous Banach lattice. Then $T\in B(X)$ iff $T:(X,\lvert\cdot\rvert,X)\to(X,\lVert\cdot\rVert_X,\mathbb{R})$ 
is sequentially $p$-continuous.
\end{prop}

\begin{proof}
$(\Rightarrow)$ Assume that $T\in B(X)$, and let $x_n\pc 0$ in $(X,\lvert\cdot\rvert,X)$. 
Then $x_n\oc 0$ in $X$. Since $(X,\lVert\cdot\rVert_X)$ is $\sigma$-order continuous Banach lattice then $x_n\nc 0$ and hence $Tx_n\nc 0$. 
Therefore, $T:(X,\lvert\cdot\rvert,X)\to(X,\lVert\cdot\rVert_X,\mathbb{R})$ is sequentially $p$-continuous.\\
$(\Leftarrow)$ Assume $T:(X,\lvert\cdot\rvert,X)\to(X,\lVert\cdot\rVert_X,\mathbb{R})$ to be sequentially $p$-continuous. 
Suppose $x_n\nc 0$ and let $x_{n_k}$ be a subsequence. Then clearly $x_{n_k}\nc 0$. Since $(X,\lVert\cdot\rVert_X)$ is a Banach lattice, 
there is a subsequence $x_{n_{k_j}}$ such that $x_{n_{k_j}}\oc 0$ in $X$ (cf. \cite[Thm.VII.2.1]{V}), and so $x_{n_{k_j}}\pc 0$ in $(X,\lvert\cdot\rvert,X)$. 
Since $T$ is sequentially $p$-continuous then $Tx_{n_{k_j}}\nc 0$. Thus, it follows from Lemma \ref{subseq criterion} that $Tx_n\nc 0$.
\end{proof}

\begin{prop}\label{seqpcontimpliesnormcont}
Let $(X,p,E)$ be an LNVL with a Banach lattice $(E,\lVert\cdot\rVert_E)$ and $(Y,m,F)$ be an LNS with a $\sigma$-order continuous normed lattice
$(F,\lVert \cdot \rVert_F)$. If $T:(X,p,E)\to(Y,m,F)$ is sequentially $p$-continuous then 
$T:(X,p\text{-}\lVert\cdot\rVert_E)$ $\to(Y,m\text{-}\lVert\cdot\rVert_F)$ is norm continuous.
\end{prop}

\begin{proof}
Let $x_n$ be a sequence in $X$ such that $x_n\xrightarrow{p\text{-}\lVert\cdot\rVert_E}0$ (i.e. $\lVert p(x_n)\rVert_E\to 0$). 
Given a subsequence $x_{n_k}$ then $\lVert p(x_{n_k})\rVert_E\to 0$. Since $(E,\lVert\cdot\rVert_E)$ is a Banach lattice, 
there is a further subsequence $x_{n_{k_j}}$ such that $p(x_{n_{k_j}})\oc 0$ in $E$ (cf. \cite[Thm.VII.2.1]{V}). Hence, $x_{n_{k_j}}\pc 0$ in $(X,p,E)$. 
Now, the $p$-continuity of $T$ implies $m(Tx_{n_{k_j}})\oc 0$ in $F$. But $F$ is  $\sigma$-order continuous and so $\lVert m(Tx_{n_{k_j}})\rVert_F\to 0$ 
or $m\text{-}\lVert Tx_{n_{k_j}}\rVert_F\to 0$. Hence, Lemma \ref{subseq criterion} implies $m\text{-}\lVert Tx_n\rVert_F\to 0$. So $T$ is norm continuous.
\end{proof}

The next theorem is a partial converse of Proposition \ref{seqpcontimpliesnormcont}. 
\begin{thm}\label{normcontimpliespcont}
Suppose $(X,p,E)$ to be an LNS with an order continuous $($respectively, $\sigma$-order continuous$)$ normed lattice $(E,\lVert\cdot\rVert_E)$ and $(Y,m,F)$ to be an LNS with an atomic Banach lattice $(F,\lVert\cdot\rVert_F)$. 
Assume further that$:$
\begin{enumerate}
\item[(i)] $T:(X,p\text{-}\lVert \cdot \rVert_E)\to(Y,m\text{-}\lVert\cdot\rVert_F)$ is norm continuous, and
\item[(ii)]$T:(X,p,E)\to(Y,m,F)$ is $p$-bounded.
\end{enumerate}
Then $T:(X,p,E)\to(Y,m,F)$ is $p$-continuous $($respectively, sequentially $p$-continuous$)$.
\end{thm}

\begin{proof}
We assume that $(E,\lVert\cdot\rVert_E)$ is an order continuous normed lattice and show the $p$-continuity of $T$, the other case is similar. Suppose $x_\alpha\pc 0$ in $(X,p,E)$ then $p(x_\alpha)\oc 0$ in $E$ and so there is $\alpha_0$ such that $p(x_\alpha)\leq e$ for all $\alpha\geq\alpha_0$. 
Thus, $(x_\alpha)_{\alpha\geq\alpha_0}$ is $p$-bounded and, since $T$ is $p$-bounded then $(Tx_\alpha)_{\alpha\geq\alpha_0}$ is $p$-bounded in $(Y,m,F)$. 
	
Since $(E,\lVert\cdot\rVert_E)$ is order continuous and $p(x_\alpha)\oc 0$ in $E$ then $\lVert p(x_\alpha)\rVert_E\to 0$ or $p\text{-}\lVert x_\alpha\rVert_E\to 0$. 
The norm continuity of $T:(X,p\text{-}\lVert\cdot\rVert_E)\to(Y,m\text{-}\lVert\cdot\rVert_F)$ ensures that $\lVert m(Tx_\alpha)\rVert_F\to 0$ or $m\text{-}\lVert Tx_\alpha\rVert_F\to 0$. 
In particular, $\lVert m(Tx_\alpha)\rVert_F\to 0$ for $\alpha\geq \alpha_0$.
	
Let $a\in F$ be an atom, and $f_a$ be the biorthogonal functional corresponding to $a$ then $ f_a\big (m(Tx_\alpha) \big )\to 0.$ Since $m(Tx_\alpha)$ is order bounded for all $\alpha\geq\alpha_0$ and $ f_a\big (m(Tx_\alpha) \big )\to 0$ for any atom $a\in F$, the atomicity of $F$ implies that $m(Tx_\alpha)\oc 0$ in $F$ as $\alpha_0\le\alpha\to\infty$. Thus, $T:(X,p,E)\to(Y,m,F)$ is $p$-continuous.
\end{proof}

The next result extends the well-known fact that every order continuous operator between vector lattices is order bounded, and its proof is similar to \cite[Thm.2.1]{YG}.

\begin{prop}\label{p-cont.isp-bdd}
Let $T$ be a $p$-continuous operator between LNSs $(X,p,E)$ and $(Y,m,F)$ then $T$ is $p$-bounded.
\end{prop}

\begin{proof}
Assume that $T:X\to Y$ is $p$-continuous. Let $A\subset X$ be $p$-bounded (i.e. there is $e\in E$ such that $p(a)\leq e$ for all $a\in A$). 
Let $I= \mathbb{N}\times A$ be an index set with the lexicographic order. That is: $(m,a')\leq(n,a)$ iff $m<n$ or else $m=n$ and $p(a')\leq p(a)$. 
Clearly, $I$ is directed upward. Define the following net as $x_{(n,a)}=\frac{1}{n}a$. Then $p(x_{(n,a)})=\frac{1}{n}p(a)\leq \frac{1}{n}e$. 
So $p(x_{(n,a)})\oc 0$ in $E$ or $x_{(n,a)}\pc 0$. By $p$-continuity of $T$, we get $m(Tx_{(n,a)})\oc 0$. So there is a net $(z_\beta)_{\beta\in B}$ 
such that $z_\beta \downarrow 0$ in $F$ and for any $\beta\in B$, there exists $(n',a')\in I$ satisfying $m(Tx_{(n,a)})\leq z_\beta$ for all $(n,a)\geq (n',a')$. 
Fix $\beta_0\in B$. Then there is $(n_0,a_0)\in I$ satisfying $m(Tx_{(n,a)})\leq z_{\beta_0}$ for all $(n,a)\geq (n_0,a_0)$. In particular, $(n_0+1,a)\geq (n_0,a_0)$ 
for all $a\in A$. Thus, $m(Tx_{(n_0+1,a)})=m(\frac {1}{n_0+1}Ta)\leq z_{\beta_0}$ or $m(Ta)\leq (n_0+1)z_{\beta_0}$ for all $a\in A$. Therefore, $T$ is $p$-bounded. 
\end{proof}

\begin{rem} \ \
\begin{enumerate}
	\item[(i)] It is known that the converse of Proposition \ref{p-cont.isp-bdd} is not true. For example, let $X=C[0,1]$ then $X^*=X^\sim$ and $X_c^\sim=X_n^\sim=\{0\}.$ Here $X_c^\sim$ denotes the $\sigma$-order continuous dual of $X$ and $X_n^\sim$ denotes the order continuous dual of $X.$ So, for any $0\neq f \in X^*$ we have $f:(X,\lvert \cdot \rvert,X)\to (\mathbb{R},\lvert \cdot \rvert,\mathbb{R})$ is $p$-bounded, which is not $p$-continuous.
	\item[(ii)] If $T:(X,E)\to(Y,F)$ between two LNVLs is $p$-continuous then $T:X\to Y$ as an operator between two vector lattices 
	need not be order bounded. Let's consider Lozanovsky's example $\big($cf. \cite[Exer.10,p.289]{AB}$\big)$. If $T:L_1[0,1]\to c_0$ is defined by 
	$$
	T(f)=\bigg ( \int_{0}^{1} f(x)sinx \ dx,\int_{0}^{1} f(x)sin2x\ dx,... \bigg).
	$$
	Then it can be shown that $T$ is norm bounded which is not order bounded. So $T: (L_1[0,1],\lVert \cdot \rVert_{L_1},\mathbb{R})\to (c_0,\lVert \cdot \rVert_{\infty},\mathbb{R})$ is $p$-continuous and $T:L_1[0,1]\to c_0$ is not order bounded.
\end{enumerate}
\end{rem}

\noindent

Recall that $T\in L(X,Y)$; where $X$ and $Y$ are normed spaces, is called \textit{Dunford-Pettis} if $x_n\wc 0$ in $X$ implies $Tx_n\norm 0$ in $Y$.

\begin{prop}\label{dunford-pettis}
Let $(X,\lVert\cdot\rVert_X)$ be a normed lattice and $(Y,\lVert\cdot\rVert_Y)$ be a normed space. 
Put $E:=\mathbb{R}^{X^*}$ and define $p:X\to E_+$ by $p(x)[f]=|f|(|x|)$ for $f\in X^*$. 
It is easy to see that $(X,p,E)$ is an LNVL $($cf. \cite[Ex.4]{AEEM}$)$.
\begin{enumerate}
\item[(i)] If $T\in L(X,Y)$ is a Dunford-Pettis operator then $T:(X,p,E)\to(Y,\lVert\cdot\rVert_Y,\mathbb{R})$ is sequentially $p$-continuous.
	
\item[(ii)] The converse holds true if the lattice operations of $X$ are weakly sequentially continuous.
\end{enumerate}
\end{prop}

\begin{proof}
(i) Assume that $x_n\pc 0$ in $X$. Then $p(x_n)\oc 0$ in $E$, and hence $p(x_n)[f]\to 0$ or $|f|(|x_n|)\to 0$ for all $f\in X^*$. 
From which, it follows that $|x_n|\wc 0$ and so $x_n\wc 0$ in $X$. Since $T$ is a Dunford-Pettis operator then $Tx_n\xrightarrow{\lVert\cdot\lVert_Y}0$.

(ii) Assume that $x_n\wc 0$. Since the lattice operations of $X$ are weakly sequentially continuous then we get $|x_n|\wc 0$. 
So, for all $f\in X^*$, we have $|f|(|x_n|)\to 0$ or $p(x_n)[f]\to 0$. Thus, $x_n\pc 0$ and, since $T$ is sequentially $p$-continuous, 
we get $Tx_n\xrightarrow{\lVert\cdot\lVert_Y}0$. Therefore, $T$ is Dunford-Pettis.
\end{proof}

\begin{rem}
It should be noticed that there are many classes of Banach lattices that satisfy condition {\em (ii)} of Proposition \ref{dunford-pettis}. For example the lattice operations of atomic order continuous Banach lattices, $AM$-spaces and Banach lattices with atomic topological dual are all weakly sequentially continuous $($see respectively, \cite[Prop. 2.5.23]{M}, \cite[Thm. 4.31]{AB} and \cite[Cor. 2.2]{AE}$)$
\end{rem}

It is known that any positive operator from a Banach lattice into a normed lattice is norm continuous or, equivalently, is norm bounded (see e.g., \cite[Thm.4.3]{AB}). 
Similarly we have the following result.

\begin{thm}\label{positivityandboundedness}
Let $(X,p,E)$ be a sequentially $p$-complete LNVL such that $(E,\lVert\cdot\rVert_E)$ is a Banach lattice, and let $(Y,\lVert\cdot\rVert_Y)$ be a normed lattice.
If $T:X\to Y$ is a positive operator then $T$ is $p$-bounded as an operator from $(X,p,E)$ into $(Y,\lVert\cdot\rVert_Y,\mathbb{R})$.
\end{thm}	

\begin{proof}
Assume that $T:(X,p,E)\to(Y,\lVert\cdot \rVert_Y,\mathbb{R})$ is not $p$-bounded. Then there is a $p$-bounded subset $A$ of $X$ such that $T(A)$ is not norm bounded in $Y$. 
Thus, there is $e\in E_+$ such that $p(a)\leq e$ for all $a\in A$, but $T(A)$ is not norm bounded in $Y$. Hence, for any $n\in  \mathbb{N}$, there is an $x_n\in A$ 
such that $\lVert Tx_n\rVert_Y \geq n^3$. Since $\lvert Tx_n \rvert \leq T\lvert x_n \rvert$, we may assume without loss of generality that $x_n\geq 0$. 
Consider the series $\sum\limits_{n=1}^{\infty}\frac{1}{n^2}x_n$ in the mixed-norm space $(X,p\text{-}\lVert\cdot\rVert_E)$, which is a Banach lattice due to Lemma \ref{mixednormisbanach}. Then
$$
  \sum\limits_{n=1}^{\infty}p\text{-}\lVert \frac{1}{n^2}x_n \rVert_E=\sum\limits_{n=1}^{\infty} \frac{1}{n^2}\lVert p(x_n)\rVert_E\leq \lVert e\rVert_E\sum\limits_{n=1}^{\infty}\frac{1}{n^2}<\infty.
$$ 
Since the series $\sum\limits_{n=1}^{\infty}\frac{1}{n^2}x_n$ is absolutely convergent, it converges to some element, say $x$, i.e. $x=\sum\limits_{n=1}^{\infty}\frac{1}{n^2}x_n\in X$. 
Clearly, $x\geq \frac{1}{n^2}x_n$ for every $n\in  \mathbb{N}$ and, since $T\geq 0$ then $T(x)\geq \frac{1}{n^2}Tx_n$, which implies 
$\lVert Tx\rVert_Y\geq\frac{1}{n^2}\lVert Tx_n\rVert_Y\geq n$ for all $n\in \mathbb{N}$; a contradiction. 
\end{proof}

\begin{exam}
$($Sequential $p$-completeness in Theorem \ref{positivityandboundedness} can not be removed$)$
\noindent	
Let $T:(c_{00},\lvert\cdot\rvert,\ell_\infty)\to(\mathbb{R},\lvert\cdot\rvert,\mathbb{R})$ be defined by $T(x_n)=\sum\limits_{n=1}^{\infty}nx_n$. 
Then $T\geq 0$ and clearly the LNVL $(c_{00},\lvert\cdot\rvert,\ell_\infty)$ is not sequentially $p$-complete.
	
Consider the $p$-bounded sequence $e_n$ in $(c_{00},\lvert\cdot\rvert,\ell_\infty)$. Since $Te_n=n$ for all $n\in \mathbb{N}$, the sequence 
$Te_n$ is not norm bounded in $\mathbb{R}$. Hence, $T$ is not $p$-bounded.
\end{exam}

\begin{exam}$($Norm completeness of $(E,\left\|\cdot\right\|_E)$ can not be removed in Theorem \ref{positivityandboundedness}$)$
\noindent
Consider the LNVL $(c_{00},p,c_{00})$, where $p(x_n)= (\sum\limits_{n=1}^{\infty}\lvert x_n\rvert)e_1$. It can be seen easily that  
$(c_{00},p,c_{00})$ is sequentially $p$-complete. Note that $(c_{00},\lVert \cdot\rVert_{\infty})$ is not norm complete. 
Define $S:(c_{00},p,c_{00})\to(\mathbb{R},\lvert\cdot\rvert,\mathbb{R})$ by $S(x_n)=\sum\limits_{n=1}^{\infty}nx_n$. 
Then $S\geq 0$, $p(e_n)\leq e_1$ for each $n\in \mathbb{N}$. But $Se_n=n$ is not bounded in $\mathbb{R}$.
\end{exam}

It is well-known that the adjoint of an order bounded operator between two vector lattices is always order bounded and order continuous 
(see, for example \cite[Thm.1.73]{AB}). The following two results deal with a similar situation.

\begin{thm}\label{adjoint}
Let $(X,\lVert\cdot\rVert_X)$ be a normed lattice and $Y$ be a vector lattice. Let $Y_c^\sim$ denote the $\sigma$-order continuous dual of $Y$. 
If $0\leq T:(X,\lVert\cdot\rVert_X,\mathbb{R})\to(Y,\lvert\cdot\rvert,Y)$ is sequentially $p$-continuous and $p$-bounded 
then the operator $T^{\sim}:(Y_c^\sim,\lvert\cdot\rvert,Y_c^\sim)\to(X^*,\lVert\cdot\rVert_{X^*},\mathbb{R})$ defined by $T^{\sim}(f):=f\circ T$ is $p$-continuous.
\end{thm}

\begin{proof}
First, we prove that $T^\sim(f)\in X^*$ for each $f\in Y_c^\sim$. Assume $x_n\norm 0$. Since $T$ is sequentially $p$-continuous then $Tx_n\oc 0$ in $Y$. 
Since $f$ is $\sigma$-order continuous then $f(T x_n)\to 0$ or $(f\circ T)(x_n)\to 0$. Hence, we have $f\circ T\in X^*$.
	
Next, we show that $T^{\sim}$ is $p$-continuous. Assume $0\leq f_\alpha\oc 0$ in $Y_c^\sim$, we show $\lVert T^{\sim} f_\alpha\rVert_{X^*}\to 0$ or $\lVert f_\alpha\circ T\rVert_{X^*}\to 0$. 
Now, $\lVert f_\alpha\circ T\rVert_{X^*}=\sup\limits_{x\in B_X}\lvert(f_\alpha \circ T)x\rvert$. Since $B_X$ is $p$-bounded in $(X,\lVert \cdot \rVert_X,\mathbb{R})$ and $T$ is $p$-bounded operator 
then $T(B_X)$ is order bounded in $Y$. So there exists $y\in Y_+$ such that $-y\leq Tx\leq y$ for all $x\in B_X$.  Hence $-f_\alpha y\leq(f_\alpha\circ T)x\leq f_\alpha y$ for all $x\in B_X$  
and for all $\alpha$.  So $\lVert f_\alpha\circ T\rVert_{X^*}\subseteq [-f_\alpha y,f_\alpha y]$ for all $\alpha$.  
It follows from \cite[Thm.VIII.2.3]{V} that $\lim\limits_{\alpha}f_\alpha y=0$. Thus, $\lim\limits_{\alpha}\lVert f_\alpha\circ T\rVert_{X^*}=0$. Therefore, $T^{\sim}$ is $p$-continuous.
\end{proof}

\begin{thm}\label{AL-adjoint}
Let $X$ be a vector lattice and $Y$ be an $AL$-space. Assume $0\leq T:(X,\lvert\cdot\rvert,X)\to(Y,\lVert\cdot\rVert_Y,\mathbb{R})$ is sequentially $p$-continuous. 
Define $T^{\sim}:(Y^*,\lVert\cdot\rVert_{Y^*},\mathbb{R})\to(X^\sim,\lvert\cdot\rvert,X^\sim)$ by $T^{\sim}(f)=f\circ T$. Then $T^{\sim}$ is sequentially $p$-continuous and $p$-bounded.
\end{thm}

\begin{proof}
Clearly, if $f\in Y^*$ then $f\circ T$ is order bounded, and so $T^{\sim}(f)\in X^\sim$.
	
We prove that $T^{\sim}$ is $p$-bounded. Let $A\subseteq Y^*$ be a $p$-bounded set in $(Y^*,\lVert\cdot\rVert_{Y^*},\mathbb{R})$ 
then there is $0<c<\infty$ such that $\lVert f\rVert_{Y^*}\leq c$ for all $f\in A$. Since $Y^*$ is an $AM$-space with a strong unit then $A$ is order bounded in $Y^*$; 
i.e., there is a $g\in Y^*_+$ such that $-g\leq f\leq g$ for all $f\in A$. That is, $-g(y)\leq f(y)\leq g(y)$ for any $y\in Y_+$, which implies $-g(Tx)\leq f(Tx)\leq g(Tx)$ for all $x\in X_+.$ 
Thus, $-g\circ T\leq f\circ T\leq g\circ T$ or $-g\circ T\leq T^{\sim}f\leq g\circ T$ for every $f\in A$. Therefore, $T^{\sim}(A)$ is $p$-bounded in $(X^\sim,\lvert\cdot\rvert,X^\sim).$
	
Next, we show that $T^{\sim}$ is sequentially $p$-continuous. Assume $0\leq f_n\xrightarrow{\left\|\cdot\right\|_{Y^*}}0$ in $(Y^*,\lVert\cdot \rVert_{Y^*})$.  
Since $Y^*$ is an $AM$-space with a strong unit, say $e$, then $f_n \xrightarrow{\lVert\cdot\rVert_e} 0$. It follows from \cite[Thm.62.4]{LZ1} that $f_n$ $e$-converges to zero in $Y^*$. 
Thus, there is a sequence $\varepsilon_k\downarrow 0$ in $\mathbb{R}$ such that for all $k\in \mathbb{N}$ there is $n_k\in  \mathbb{N}$ satisfying $f_n\leq\varepsilon_ke$ for all $n\geq n_k$. In particular, $f_n(Tx)\leq\varepsilon_ke(Tx)$ 
for all $x\in X_+$ and for all $n\geq n_k$. From which it follows that $f_n\circ T$ $e$-converges to zero in $X^\sim$ and so $f_n\circ T\oc 0$ in $X^\sim$. 
Hence, $T^{\sim}(f_n)\oc 0$ in $X^\sim$ and $T^{\sim}$ is sequentially $p$-continuous.
\end{proof}

\section{$p$-Compact Operators}

Given normed spaces $X$ and $Y$. Recall that $T\in L(X,Y)$ is said to be compact if $T(B_X)$ is relatively compact in $Y$. 
Equivalently, $T$ is compact iff for any norm bounded sequence $x_n$ in $X$ there is a subsequence $x_{n_k}$ such that the sequence $Tx_{n_k}$ is convergent in $Y$. 
Motivated by this, we introduce the following notions.

\begin{definition}
Let $X$, $Y$ be two LNSs and $T\in L(X,Y)$. Then
\begin{enumerate}
\item[(1)] $T$ is called {\em $p$-compact} if, for any $p$-bounded net $x_\alpha$ in $X$, 
there is a subnet $x_{\alpha_\beta}$ such that $Tx_{\alpha_\beta}\pc y$ in $Y$ for some $y\in Y$.
\item[(2)] $T$ is called {\em sequentially $p$-compact} if, for any $p$-bounded sequence $x_n$ in $X$, 
there is a subsequence $x_{n_k}$ such that $Tx_{n_k}\pc y$ in  $Y$ for some $y\in Y$.
\end{enumerate}
\end{definition}

\begin{exam}\label{seq.p-compact not p-compact}
$($A sequentially $p$-compact operator need not be $p$-compact$)$\\
Let's take the vector lattice
 $$
 c_{\aleph_1} (\mathbb{R}):=\big\{f:\mathbb{R}\to \mathbb{R}:\exists a\in \mathbb{R}, \forall \varepsilon >0, \ \mathbf{card}\big(\{x\in\mathbb{R}: \lvert f(x)-a\rvert\geq \varepsilon\}\big)<{\aleph_1}\big\}.
 $$
Consider the identity operator $I$ on $\big(c_{\aleph_1}(\mathbb{R}),\lvert \cdot \rvert,c_{\aleph_1}(\mathbb{R})\big)$. 
Let $f_n$ be a $p$-bounded sequence in $\big(c_{\aleph_1}(\mathbb{R}),\lvert \cdot \rvert,c_{\aleph_1}(\mathbb{R})\big)$. So there is $g\in c_{\aleph_1}(\mathbb{R})$ such that $0\leq f_n\leq g$ for all $n\in  \mathbb{N}.$\\
For any $n\in \mathbb{N}$, there is $a_n\in \mathbb{R}_+$ such that for all $\varepsilon>0$, $\mathbf{card}\big(\{x\in\mathbb{R}: \lvert f(x)-a_n\rvert\geq \varepsilon\}\big)<{\aleph_1}.$ Clearly the sequence $a_n$ is bounded in $\mathbb{R},$ so there is a subsequence $a_{n_k}$ and $a\in \mathbb{R}$ such that $a_{n_k}\to a$ as $k\to \infty.$ For each $m, k \in  \mathbb{N}$, let $A_{m,n_k}:=\{x\in \mathbb{R}:\lvert f_{n_k}(x)-a_{n_k}\rvert\geq \frac{1}{m}\}.$ Put $A=\bigcup\limits_{m=1}^{\infty}\bigcup\limits_{k=1}^{\infty}A_{m,n_k}$ and let $h=a\chi_{\mathbb{R}\backslash A}$ then $f_{n_k}\oc h$, since order convergence in $c_{\aleph_1}(\mathbb{R})$ is pointwise convergence. Thus, $I$ is sequentially $p$-compact.

On the other hand; let $\mathcal{F}(\mathbb{R})$ be the collection of all finite subsets of $\mathbb{R}$. For each $\alpha \in \mathcal{F}(\mathbb{R})$ let $f_\alpha:= \chi_{\mathbb{R}\backslash \alpha}$. 
Then $f_\alpha\leq \mathds 1 \in c_{\aleph_1}(\mathbb{R})$ and $a_\alpha =1$.  But, for every subnet $f_{\alpha_\beta}$, we have  $f_{\alpha_\beta}(x) \not \to 1$ for any $x \in \mathbb{R}$, so $f_{\alpha_\beta}$ does not converge in order to $\mathds 1$. Therefore, $I$ is not $p$-compact.
\end{exam}

\noindent

In connection with Example \ref{seq.p-compact not p-compact} the following question arises naturally.
\begin{question}
Is it true that every $p$-compact operator is sequentially $p$-compact?
\end{question}

\noindent
\begin{definition}
Let $X$, $Y$ be two LNSs and $T\in L(X,Y)$. Then
\begin{enumerate}
	\item[(1)] $T$ is called {\em $rp$-compact}, if for any $p$-bounded net $x_\alpha$ in $X$, 
	there is a subnet $x_{\alpha_\beta}$ such that $Tx_{\alpha_\beta}\rp y$ in $Y$ for some $y\in Y.$
	\item[(2)] $T$ is called {\em sequentially $rp$-compact}, if for any $p$-bounded sequence $x_n$ in $X$, there is a subsequence $x_{n_k}$ such that $Tx_{n_k}\rp y$ in  $Y$ for some $y\in Y$.
	\end{enumerate}
\end{definition}

\noindent
\begin{rem}\ \
\begin{enumerate}
	\item[(i)] Every $($sequentially$)$ $rp$-compact is $($sequentially$)$ $p$-compact.
	\item[(ii)] The converse of {\em (i)} in the sequential case need not to be true. Consider the identity operator $I$ on $(\ell_\infty,\lvert \cdot \rvert,\ell_\infty)$. 
It can be easily seen that $I$ is sequentially $p$-compact but is not sequentially $rp$-compact. 
	\item[(iii)] We do not know whether or not every $rp$-compact operator is sequentially $rp$-compact and whether or not the vice versa is true.
\end{enumerate}
\end{rem}

In the following example we show that $p$-compact operators generalize many well-known classes of operators.  

\noindent
\begin{exam}\label{examplesofp-compact}\ \ 
\begin{enumerate}

\item[(i)] Let $(X,\lVert\cdot\rVert_X)$ and $(Y,\lVert\cdot\rVert_Y)$ be normed spaces. 
Then $T:(X,\lVert\cdot\rVert_X,\mathbb{R})\to(Y,\lVert\cdot\rVert_Y,\mathbb{R})$ is $($sequentially$)$ $p$-compact 
iff $T:X\to Y$ is compact. 

\item[(ii)] Let $X$ be a vector lattice and $Y$ be a normed space. An operator $T\in L(X,Y)$ is said to be $AM$-compact if $T[-x,x]$ is relatively compact for every $x\in X_+$ $($cf. \cite[Def.3.7.1]{M}$)$. 
Therefore, $T\in L(X,Y)$ is $AM$-compact operator iff $T:(X,\lvert\cdot\rvert,X)\to(Y,\lVert\cdot\rVert_Y,\mathbb{R})$ is $p$-compact.

\item[(iii)] Let $X$ and $Y$ be normed spaces. An operator $T\in L(X,Y)$ is said to be weakly compact if $T(B_X)$ is relatively weakly compact.\\
Let $X$ be a normed space and $(Y,\lVert\cdot\rVert_Y)$ be a normed lattice. Let $E:=\mathbb{R}^{Y^*}$ and consider the LNVL $(Y,p,E)$, 
where $p(y)[f]=|f|(|y|)$ for all $f\in Y^*$. Then $T\in L(X,Y)$ is weakly compact iff $T:(X,\lVert\cdot\rVert_X,\mathbb{R})\to(Y,p,E)$ is sequentially $p$-compact.

\item[(iv)] Let $X$ be a vector lattice and $Y$ be a normed space. An operator $T\in L(X,Y)$ is said to be order weakly compact 
if $T[-x,x]$ is relatively weakly compact for every $x\in X_+$ $($cf. \cite[Def.3.4.1.ii)]{M}$)$.\\ 
Let $X$ be a vector lattice and $(Y,\lVert\cdot\rVert_Y)$ be a normed lattice. Let $E:=\mathbb{R}^{Y^*}$ and consider the LNVL $(Y,p,E)$,  
where $p(y)[f]=|f|(|y|)$ for all $f\in Y^*$. Then $T\in L(X,Y)$ is order weakly compact iff $T:(X,\lvert\cdot\rvert,X)\to(Y,p,E)$ is sequentially $p$-compact.
\end{enumerate}
\end{exam}

\begin{rem}
It is known that any compact operator is norm continuous, but in general we may have a $p$-compact operator which is not $p$-continuous. 
Indeed, consider the following example taken from \cite{MMP}. Denote by $\mathcal{B}$ the Boolean algebra of the Borel subsets of $[0,1]$ equals up to measure null sets. 
Let $\mathcal{U}$ be any ultrafilter on $\mathcal{B}$. Then it can be shown that the linear operator $\varphi_{\mathcal{U}}:L_\infty[0,1]\to\mathbb{R}$ defined by 
$$ 
  \varphi_{\mathcal{U}}(f):=\lim\limits_{A\in\mathcal{U}}\frac{1}{\mu(A)}\int_{A}fd\mu
$$ 
is $AM$-compact which is not order-to-norm continuous; see \cite[Ex.4.2]{MMP}.  
That is, the operator $\varphi_{\mathcal{U}}:(L_\infty[0,1],\lvert\cdot\rvert,L_\infty[0,1])\to(\mathbb{R},\lvert\cdot\rvert,\mathbb{R})$  is $p$-compact, which is not $p$-continuous.
\end{rem}

\begin{exam}
$($A sequentially $p$-compact operator need not be $p$-bounded$)$\\
Let's consider again Lozanovsky's example $\big($cf. \cite[Exer.10,p.289]{AB}$\big)$. If $T:L_1[0,1]\to c_0$ is defined by 
$$
T(f)=\bigg ( \int_{0}^{1} f(x)sinx\ dx,\int_{0}^{1} f(x)sin2x\ dx,... \bigg).
$$
Then it can be shown that $T$ is not order bounded. So $T$ is not $p$-bounded as an operator from the LNS $\big(L_1[0,1],\lvert \cdot \rvert,L_1[0,1]\big)$ into the LNS $\big(c_0,\lvert \cdot \rvert,c_0\big).$

On the other hand, let $f_n$ be a $p$-bounded sequence in $\big(L_1[0,1],\lvert \cdot \rvert,L_1[0,1]\big)$ then $f_n$ is order bounded in $L_1[0,1]$. By a standard diagonal argument there are a subsequence $f_{n_k}$ and a sequence $a=(a_k)_{k\in  \mathbb{N}}\in c_0$ such that $Tf_{n_k}\oc a$ in $c_0.$ Therefore, $T:\big (L_1[0,1],\lvert \cdot \rvert,L_1[0,1] \big )\to (c_0,\lvert \cdot \rvert,c_0)$ is sequentially $p$-compact.
\end{exam}

Since any compact operator is norm bounded, the following question arises naturally.

\begin{question} \label{Q1}
	Is it true that every $p$-compact operator is $p$-bounded?
\end{question}
Regarding (sequentially) $rp$-compact operators, we have the following.
\begin{question}\ \
	\begin{enumerate}
		\item[(1)] Is it true that every $rp$-compact operator is $p$-bounded or equivalently $rp$-continuous?
		\item[(2)] Is it true that every sequentially $rp$-compact operator is $p$-bounded?
	\end{enumerate}
\end{question}

Let $(X,E)$ be a decomposable LNS and $(Y,F)$ be an LNS such that $F$ is order complete then, by \cite[4.1.2,p.142]{K}, each dominated operator $T:X\to Y$ has the exact dominant $\pmb{\lvert} T \pmb{\rvert}$. 
Therefore, the triple $\big(M(X,Y),p,L^\sim(E,F)\big)$ is an LNS, where $p:M(X,Y)\to L_+^\sim(E,F)$ is defined by $p(T)=\pmb{\lvert}T\pmb{\rvert}$ (see, for example \cite[4.2.1,p.150]{K}). Thus, if $T_\alpha$ is a net in $M(X,Y)$ then  $T_\alpha\pc  T$ in $M(X,Y)$, whenever $\pmb{\lvert} T_\alpha-T\pmb{\rvert}\oc 0$ in $L^\sim(E,F)$.

\begin{thm}
Let $(X,p,E)$ be a decomposable LNS and $(Y,q,F)$ be a sequentially $p$-complete LNS such that $F$ is order complete. 
If $T_m$ is a sequence in $M(X,Y)$ and each $T_m$ is sequentially $p$-compact with $T_m\pc T$ in $M(X,Y)$ then $T$ is sequentially $p$-compact. 
\end{thm}

\begin{proof}
Let $x_n$ be a $p$-bounded sequence in $X$ then there is $e\in E_+$ such that $p(x_n)\leq e$ for all $n\in \mathbb{N}$. 
By a standard diagonal argument, there exists a subsequence $x_{n_k}$ such that for any $m\in \mathbb{N}$, $T_mx_{n_k}\pc y_m$ for some $y_m\in Y$. 

We show that $y_m$ is a $p$-Cauchy sequence in $Y$. 
\begin{eqnarray*}
q(y_m-y_j)&=&q(y_m-T_mx_{n_k}+T_mx_{n_k}-T_jx_{n_k}+T_jx_{n_k}-y_j)\\ &\leq& q(y_m-T_mx_{n_k})+q(T_mx_{n_k}-T_jx_{n_k})+q(T_jx_{n_k}-y_j).
\end{eqnarray*}
The first and the third terms in the last inequality both order converge to zero as $m\to\infty$ and $j\to\infty$, respectively. Since $T_m\in M(X,Y)$ for all $m\in \mathbb{N}$ then
$$
  q(T_mx_{n_k}-T_jx_{n_k})\leq\pmb{\lvert} T_m-T_j\pmb{\rvert}(p(x_{n_k}))\leq\pmb{\lvert} T_m-T_j\pmb{\rvert} (e).
$$
Since $T_m\pc T$ in $M(X,Y)$ then, by \cite[Thm.VIII.2.3]{V}, it follows that $\pmb{\lvert}T_m-T_j\pmb{\rvert}(e)\oc 0$ in $F$, as $m,j\to\infty$. 
Thus, $q(y_m-y_j)\oc 0$ in $F$ as $m,j\to\infty$. Therefore, $y_m$ is $p$-Cauchy. 
Since $Y$ is sequentially $p$-complete then there is $y\in Y$ such that $q(y_m-y)\oc 0$ in $F$ as $m\to\infty$. Hence,
\begin{eqnarray*}
q(Tx_{n_k}-y)&\leq& q(Tx_{n_k}-T_mx_{n_k})+q(T_mx_{n_k}-y_m)+q(y_m-y)\\ &\leq& \pmb{\lvert} T_m-T 
\pmb{\rvert}(p(x_{n_k}))+q(T_mx_{n_k}-y_m)+q(y_m-y)\\ &\leq& \pmb{\lvert} T_m-T\pmb{\rvert} (e)+q(T_mx_{n_k}-y_m)+q(y_m-y).
\end{eqnarray*}
Fix $m\in \mathbb{N}$ and let $k\to\infty$ then 
$$
  \limsup\limits_{k\to\infty}q(Tx_{n_k}-y)\leq\pmb{\lvert}T_m-T\pmb{\rvert} (e)+q(y_m-y).
$$ 
But $m\in \mathbb{N}$ is arbitrary, so $\limsup\limits_{k\to\infty}q(Tx_{n_k}-y)=0$. Hence,
$q(Tx_{n_k}-y)\oc 0$. Therefore, $T$ is sequentially $p$-compact.
\end{proof}

\begin{prop}\label{leftandrightmultiplication}
Let $(X,p,E)$ be an LNS and $R,T,S \in L(X)$.
\begin{enumerate}
\item[(i)] If $T$ is $($sequentially$)$  $p$-compact and $S$ is $($sequentially$)$ $p$-continuous then $S\circ T$ is $($sequentially$)$ $p$-compact.
\item[(ii)] If $T$ is $($sequentially$)$ $p$-compact and $R$ is $p$-bounded then $T\circ R$ is $($sequentially$)$ $p$-compact.
\end{enumerate}
\end{prop}
 
\begin{proof}
(i) Assume $x_\alpha$ to be a $p$-bounded net in $X$. Since $T$ is $p$-compact, there are a subnet $x_{\alpha_{\beta}}$ and $x\in X$ such that $p(Tx_{\alpha_{\beta}}-x)\oc 0$. It follows from the $p$-continuity of $S$ that $p\big(S(Tx_{\alpha_{\beta}})-Sx\big)\oc 0$. Therefore, $S\circ T$ is $p$-compact.
 
(ii) Assume $x_\alpha$ to be a $p$-bounded net in $X$. Since $R$ is $p$-bounded then $Rx_\alpha$ is $p$-bounded. 
Now, the $p$-compactness of $T$ implies that there are a subnet  $x_{\alpha_{\beta}}$ and $z\in X$ such that $p\big(T(Rx_{\alpha_{\beta}})-z\big)\oc 0$. Therefore, $T\circ R$ is $p$-compact.
 
The sequential case is analogous.
\end{proof}

\begin{prop}\label{compactimpliesseq.p-compact}
Let $(X,p,E)$ be an LNS, where $(E,\lVert\cdot\rVert_E)$ is a normed lattice and $(Y,m,F)$ be an LNS, where $(F,\lVert\cdot\rVert_F)$ is a Banach lattice. 
If $T:(X,p\text{-}\lVert\cdot\rVert_E)\to(Y,m\text{-}\lVert\cdot\rVert_F)$ is compact then $T:(X,p,E)\to(Y,m,F)$ is sequentially $p$-compact.  
\end{prop}

\begin{proof}
Let $x_n$ be a $p$-bounded sequence in $(X,p,E)$. Then there is $e\in E$ such that $p(x_n)\leq e$ for all $n\in \mathbb{N}$. So $\lVert p(x_n)\rVert_E\leq\lVert e\rVert_E<\infty$. 
Hence, $x_n$ is norm bounded in $(X,p\text{-}\lVert\cdot\rVert_E)$. Since $T$ is compact then there are a subsequence $x_{n_{k}}$ and $y\in Y$ 
such that $m\text{-}\lVert Tx_{n_{k}}-y\rVert_F\to 0$ or $\lVert m(Tx_{n_{k}}-y)\rVert_F\to 0$. Since $(F,\lVert\cdot\rVert_F)$ is a Banach lattice 
then, by \cite[Thm.VII.2.1]{V} there is a further subsequence $x_{n_{k_j}}$ such that $m(Tx_{n_{k_j}}-y)\oc 0$. Therefore, $T:(X,p,E)\to(Y,m,F)$ is sequentially $p$-compact.   
\end{proof}

\begin{prop}\label{seq.p-compactimpliescompact}
Let $(X,p,E)$ be an LNS, where $(E,\lVert\cdot\rVert_E)$ is an $AM$-space with a strong unit. Let $(Y,m,F)$ be an LNS, where $(F,\lVert\cdot\rVert_F)$ is an order continuous normed lattice. 
If $T:(X,p,E)\to(Y,m,F)$ is sequentially $p$-compact then $T:(X,p\text{-}\lVert\cdot\rVert_E)\to(Y,m\text{-}\lVert\cdot\rVert_F)$ is compact.
\end{prop}

\begin{proof}
Let $x_n$ be a normed bounded sequence in $(X,p\text{-}\lVert\cdot\rVert_E)$. That is: $p\text{-}\lVert x_n\rVert_E=\lVert p(x_n)\rVert_E\leq k<\infty$ for all $n\in \mathbb{N}$. 
Since  $(E,\lVert\cdot\rVert_E)$ is an $AM$-space with a strong unit then $p(x_n)$ is order bounded in $E$. Thus, $x_n$ is a $p$-bounded sequence in  $(X,p,E)$. 
Since $T$ is sequentially $p$-compact, there are a subsequence $x_{n_{k}}$ and $y\in Y$ such that $m(Tx_{n_{k}}-y)\oc 0$ in $F$. 
Since $(F,\lVert\cdot\rVert_F)$ is order continuous then $\lVert m(Tx_{n_{k}}-y)\rVert_F\to 0$ or $m\text{-}\lVert Tx_{n_{k}}-y\rVert_F\to 0$. 
Thus, the operator $T:(X,p\text{-}\lVert\cdot\rVert_E)\to(Y,m\text{-}\lVert\cdot\rVert_F)$ is compact.
\end{proof}

The following result could be known but since we do not have a reference for it we include a proof for the sake of completeness.

\begin{lem}\label{atomipointwiseconvergence}
Let $X$ be an atomic vector lattice. Then a net $x_\alpha$ is $uo$-null iff it is pointwise null, $($that is, $\lvert x_\alpha \rvert \wedge a \oc 0$ for all atoms in $X$$)$. 
\end{lem}

\begin{proof}
The forward implication is trivial. \\
For the converse, let $x_\alpha$ be a pointwise null net in $X$. Without loss of generality, we may assume that $x_\alpha\geq 0$. Take $u\in X_+$. Then we need to show that $x_\alpha \wedge u\oc 0$. 
Consider the following directed set $\Delta =\mathcal{P}_{fin}(\Omega)\times\mathbb{N}$, where $\Omega$ is the collection of all atoms in $X$. For each $\delta=(F,n)\in\Delta$, put $y_\delta=\frac{1}{n}\sum\limits_{a\in F}a+\sum\limits_{a\in \Omega\setminus F}P_au$, where $P_a$ denotes the band projection onto 
$span\{a\}$. It is easy to see that $y_\delta\downarrow 0$ and for any $\delta \in \Delta$ there is an $\alpha_\delta$ such that for any $\alpha\geq \alpha_\delta$ we have that 
$0\leq x_\alpha \wedge u\leq y_\delta$. Therefore, $x_\alpha \wedge u\oc 0$.
\end{proof}

\begin{rem}\label{atomic}
If $X$ is an atomic $KB$-space then every order bounded net has an order convergent subnet. Indeed, let $x_\alpha$ be an order bounded net in $X$. 
Then clearly $x_\alpha$ is norm bounded and so, by \cite[Thm.7.5]{KMT} there is a subnet $x_{\alpha_{\beta}}$ such that $x_{\alpha_{\beta}}\unc x$ for some $x\in X$. 
But, in atomic order continuous Banach lattices $un$-convergence coincides with pointwise convergence $($see \cite[Cor. 4.14]{KMT}$)$. Therefore, by 
Lemma \ref{atomipointwiseconvergence} $x_{\alpha_\beta}\uoc x$. Thus $x_{\alpha_\beta}\oc x$, since $x_\alpha$ is order bounded. 
\end{rem}

\begin{prop}\label{orderboundedimpliesp-compactness}
Let $X$ be a vector lattice and $(Y,m,F)$ be an $op$-continuous LNVL such that $Y$ is atomic $KB$-space. If $T\in L^\sim(X,Y)$ then $T:(X,\lvert\cdot\rvert,X)$ $\to(Y,m,F)$ is $p$-compact.
\end{prop}

\begin{proof}
Let $x_\alpha$ be a $p$-bounded net in $(X,\lvert\cdot\rvert,X)$ then $x_\alpha$ is order bounded in $X$. Since $T$ is order bounded then $Tx_\alpha$ is order bounded in $Y$, 
which is an atomic $KB$-space. So, by Remark \ref{atomic}, there are a subnet $x_{\alpha_{\beta}}$ and $y\in Y$ such that $Tx_{\alpha_{\beta}}\oc y$. Since $(Y,m,F)$ is $op$-continuous 
then $m(Tx_{\alpha_{\beta}}-y)\oc 0$. Thus, $T$ is $p$-compact.
\end{proof} 

\begin{prop}\label{p-bdd implies p-compact}
Let $(X,p,E)$ and $(Y,\lvert\cdot\rvert,Y)$ be two LNVLs such that $Y$ is an atomic $KB$-space. If $T:(X,p,E)\to(Y,\lvert\cdot\rvert,Y)$ is $p$-bounded then $T$ is $p$-compact.
\end{prop}

\begin{proof}
Let $x_\alpha$ be a $p$-bounded net in $X$. Since $T$ is $p$-bounded then $Tx_\alpha$ is order bounded in $Y$. Since $Y$ is an atomic $KB$-space 
then, by Remark \ref{atomic}, there is a subnet $x_{\alpha_\beta}$ such that $Tx_{\alpha_\beta}\oc y$ for some $y\in Y$. Therefore, $T$ is $p$-compact.
\end{proof}

\begin{rem}\ \
\begin{enumerate}
\item[(i)] We can not omit the atomicity in Propositions \ref{orderboundedimpliesp-compactness} and \ref{p-bdd implies p-compact}$;$ 
consider the identity operator $I$ on $(L_{1}[0,1],\lvert\cdot\rvert,L_1[0,1])$ 
then the sequence of Rademacher functions is order bounded and has no order convergent subsequence, so $I$ is not $p$-compact.
\item[(ii)]	The identity operator $I$ on $(\ell_1,\lvert\cdot\rvert,\ell_1)$ satisfies 
the conditions of Proposition \ref{orderboundedimpliesp-compactness}, so $I$ is $p$-compact. 
This shows that the identity operator on an infinite dimensional space can be $p$-compact.
\item[(iii)] We do not know whether or not the identity operator $I$ on the LNS $(L_\infty[0,1],\lvert\cdot\rvert,L_\infty[0,1])$ could be $p$-compact or sequentially $p$-compact.
\end{enumerate}
\end{rem}

\begin{prop}\label{rankoneoperator}
Let $(X,p,E)$ and $(Y,m,F)$ be LNSs. Let $T:(X,p,E)\to(Y,m,F)$ be a $p$-bounded finite rank operator. Then $T$ is $p$-compact.
\end{prop}

\begin{proof}
Without lost of generality, we may suppose that $T$ is given by $Tx=f(x)y_0$ for some $p$-bounded functional $f:(X,p,E)\to(\mathbb{R},\lvert \cdot \rvert,\mathbb{R})$
and $y_0\in Y$.

Let $x_\alpha$ be a $p$-bounded net in $X$ then $f(x_\alpha)$ is bounded in $\mathbb{R}$, so there is a subnet $x_{\alpha_{\beta}}$ such that $f(x_{\alpha_{\beta}})\to\lambda$ for some $\lambda\in\mathbb{R}$. 
Now, $m(Tx_{\alpha_{\beta}}-\lambda y_0)=m\big((fx_{\alpha_{\beta}}-\lambda)y_0\big)=\lvert f(x_{\alpha_{\beta}})-\lambda\rvert m(y_0)\oc 0$ in $F$. Thus, $T$ is $p$-compact.
\end{proof}

\begin{exam}
$($The $p$-boundedness of $T$ in Proposition \ref{rankoneoperator} can not be removed$)$
\noindent	
Let $(X,p,E)$ be an LNS and $f:(X,p,E)\to(\mathbb{R},\lvert\cdot\rvert,\mathbb{R})$ be a linear functional which is not $p$-bounded. 
Then there is a $p$-bounded sequence $x_n$ such that $\lvert f(x_n)\rvert\geq n$ for all $n\in \mathbb{N}$. 
Therefore, any rank one operator $T:(X,p,E)\to(Y,m,F)$ given by the rule $Tx=f(x)y_0$, where $0\ne y_0\in Y$, is not $p$-compact.
\end{exam}

\noindent
Recall that: 
\begin{enumerate}
\item[(1)] A subset $A$ of a normed lattice $(X,\lVert\cdot\rVert)$ is called \textit{almost order bounded} if, for any $\varepsilon>0$, there is $u_\varepsilon\in X_+$ such that 
$$
  \lVert (\lvert x\rvert-u_\varepsilon)^+\rVert=\lVert\lvert x\rvert-u_\varepsilon\wedge\lvert x\rvert\rVert\leq\varepsilon \ \ \ \ (\forall x\in A).
$$  
\item[(2)] Given an LNVL $(X,p,E)$. A subset $A$ of $X$ is said to be {\em $p$-almost order bounded} if, for any $w\in E_+$, there is $x_w\in X$ such that 
$$
  p\big((\lvert x\rvert-x_w)^+\big)=p(\lvert x\rvert-x_w\wedge\lvert x\rvert)\leq w  \ \ \ \ (\forall x\in A),
$$ 
see \cite[Def.7]{AEEM}.
\noindent
If $(X,\lVert\cdot\rVert)$ is a normed lattice then a subset $A$ of $X$ is $p$-almost order bounded in $(X,\lVert\cdot\rVert,\mathbb{R})$ iff $A$ is almost order bounded in $X$. 
On the other hand, if $X$ is a vector lattice, a subset in $(X,\lvert\cdot\rvert,X)$ is $p$-almost order bounded iff it is order bounded in $X$. 
\item[(3)] An operator $T\in L(X,Y)$, where $X$ is a normed space and $Y$ is a normed lattice, is called \textit{semicompact} if $T(B_X)$ is almost order bounded in $Y$.
\end{enumerate}

\begin{definition}
Let $(X,E)$ be an LNS and $(Y,F)$ be an LNVL. A linear operator $T:X\to Y$ is called {\em $p$-semicompact} if, for any $p$-bounded set $A$ in $X$, we have that $T(A)$ is $p$-almost order bounded in $Y$.
\end{definition}

\begin{rem}\ \
\begin{enumerate}
\item[(i)] Any $p$-semicompact operator is $p$-bounded operator.
\item[(ii)] Let $T,S\in L(X)$, where $X$ is an LNS. If $T$ is $p$-semicompact and $S$ is $p$-compact then it follows easily from Proposition \ref{leftandrightmultiplication} {\em(ii)}, that $S\circ T$ is $p$-compact.
\item[(iii)] Given $T\in L(X,Y)$; where $X$ is a normed space and $Y$ is a normed lattice. 
Then $T$ is semicompact iff  $T:(X,\lVert\cdot\rVert_X,\mathbb{R})\to(Y,\lVert\cdot\rVert_Y,\mathbb{R})$ is $p$-semicompact.
\item[(iv)] For vector lattices $X$ and $Y$, we have $T\in L^\sim (X,Y)$ iff  $T:(X,\lvert\cdot\rvert,X)$ $\to(Y,\lvert\cdot\rvert,Y)$ is $p$-semicompact.
\end{enumerate}
\end{rem}

\begin{prop}
Let $(X,p,E)$ be an LNS with an $AM$-space $(E,\lVert\cdot\rVert_E)$ possessing a strong unit and $(Y,m,F)$ be an LNVL with a normed lattice $(F,\lVert\cdot\rVert_F)$. 
If $T:(X,p,E)\to(Y,m,F)$ is $p$-semicompact then $T:(X,p\text{-}\lVert\cdot\rVert_E)\to(Y,m\text{-}\lVert\cdot\rVert_F)$ is semicompact.
\end{prop}

\begin{proof}
Consider the closed unit ball $B_X$ of $(X,p\text{-}\lVert\cdot\rVert_E)$. Then $p\text{-}\lVert x\rVert_E\leq 1$ or  $\lVert p(x)\rVert_E\leq 1$ for all $x\in B_X$. 
We show that $T(B_X)$ is almost order bounded in $(Y,m\text{-}\lVert\cdot\rVert_F)$. Given $\varepsilon >0$. Let $w\in F_+$ such that 
\begin{equation}\label{epsilon}
\lVert w\rVert_F=\varepsilon.
\end{equation}
Since $\lVert p(x)\rVert_E\leq 1$ for all $x\in B_X$ and $(E,\lVert\cdot\rVert_E)$ is an $AM$-space with a strong unit, there exists $e\in E_+$ such that $p(x)\leq e$ for all $x\in B_X$. 
Thus, $B_X$ is $p$-bounded in $(X,p,E)$ and, since $T$ is $p$-semicompact, we get that $T(B_X)$ is $p$-almost order bounded in $(Y,m,F)$. So, for $w\in F_+$ in (\ref{epsilon}), 
there is $y_w\in Y_+$ such that $m\big((\lvert Tx\rvert -y_w)^+\big)\leq w$ for all $x\in B_X$, 
which implies that $\rVert m\big((\lvert Tx\rvert -y_w)^+\big)\rVert_F\leq\lVert w\rVert_F$ for all $x\in B_X$. 
Hence, $m\text{-}\rVert(\lvert Tx\rvert -y_w)^+\rVert_F\leq\varepsilon$ for all $x\in B_X$. Therefore, $T$ is semicompact.
\end{proof}

\begin{prop}
Let $(X,p,E)$ and $(Y,m,F)$ be two LNVLs. Suppose a positive linear operator $T:X\to Y$ to be $p$-semicompact. If $0\leq S\leq T$ then $S$ is $p$-semicompact.
\end{prop}

\begin{proof}
Let $A$ be a $p$-bounded set in $X$. Put $\lvert A\rvert:=\{\lvert a\rvert:a\in A\}$. Clearly $\lvert A\rvert$ is $p$-bounded.
Since $T$ is $p$-semicompact then $T(\lvert A\rvert)$ is $p$-almost order bounded. Given $w\in F_+$, there is $y_w\in Y_+$ such that 
$$
  m\big((T\lvert a\rvert-y_w)^+\big)\leq w \ \ \ \ \ \  (a\in A).
$$
Thus, for any $a\in A$,
$$
  S\lvert a\rvert\leq T\lvert a\rvert\Rightarrow(S\lvert a\rvert-y_w)^+\leq (T\lvert a\rvert-y_w)^+\Rightarrow m\big((S\lvert a\rvert-y_w)^+\big)\leq w
$$
Since $(\lvert Sa\rvert - y_w)^+\leq (S\lvert a\rvert -y_w)^+$, we have 
\begin{eqnarray*}
m\big((\lvert Sa\rvert - y_w)^+\big)\leq m\big((S\lvert a\rvert -y_w)^+\big)\leq w \ \ \ \ (\forall a\in A).
\end{eqnarray*}
Therefore, $S(A)$ is $p$-almost order bounded, and $S$ is $p$-semicompact.
\end{proof}

A linear operator $T$ from an LNS $(X,E)$ to a Banach space $(Y,\lVert\cdot\rVert_Y)$ is called \textit{generalized $AM$-compact} or \textit{$GAM$-compact} if, for any $p$-bounded set $A$ in $X$, 
$T(A)$ is relatively compact in $(Y,\lVert\cdot\rVert_Y)$; see \cite[p.1281]{P}. Clearly, $T:(X,p,E)\to(Y,\lVert\cdot\rVert_Y,\mathbb{R})$ is $GAM$-compact iff it is (sequentially) $p$-compact.

\begin{prop}\label{GAM-compact}
Let $(X,p,E)$ be an LNS and $(Y,m,F)$ be an $op$-continuous LNVL with a norming Banach lattice $(Y,\lVert\cdot\rVert_Y)$. If $T:(X,p,E)\to(Y,\lVert\cdot\rVert_Y)$ is $GAM$-compact 
then $T:(X,p,E)\to(Y,m,F)$ is sequentially $p$-compact.
\end{prop}

\begin{proof}
Let $x_n$ be a $p$-bounded sequence in $X$. Since $T$ is $GAM$-compact then there are a subsequence $x_{n_k}$ and some $y\in Y$ such that  $\lVert Tx_{n_k}-y\rVert_Y\to 0$. 
As $(Y,\lVert\cdot\rVert_Y)$ is Banach lattice then, by \cite[Thm.VII.2.1]{V}, there is a subsequence $x_{n_{k_j}}$ such that $Tx_{n_{k_j}}\oc y$ in $Y$. 
Then, by $op$-continuity of  $(Y,m,F)$, we get $Tx_{n_{k_j}}\pc y$ in $Y$. Hence, $T$ is sequentially $p$-compact.
\end{proof}

In particular, if $(X,p,E)$ is an LNS, $(Y,\lVert\cdot\rVert_Y)$ is a Banach lattice and $T:(X,p,E)\to(Y,\lVert\cdot\rVert_Y)$ is $GAM$-compact operator 
then, since $(Y,\lvert\cdot\rvert,Y)$ is always $op$-continuous LNVL, we get that $T:(X,p,E)\to(Y,\lvert\cdot\rvert,Y)$ is sequentially $p$-compact.

It is known that any compact operator is semicompact. So, the following question arises naturally.
\begin{question}\label{Q2}
Is it true that every $p$-compact operator is $p$-semicompact?
\end{question}

It should be noticed that, if Question \ref{Q1} has a negative answer then Question \ref{Q2} has a negative answer as well, since every $p$-semicompact operator is $p$-bounded, and if Question \ref{Q1} has a positive answer then every $p$-compact operator $T:(X,\lvert\cdot\rvert,X)\to(Y,\lvert\cdot\rvert,Y)$ is $p$-semicompact, where $X$ and $Y$ are vector lattices.

The converse of Question \ref{Q2} is known to be false. For instance, the identity operator $I$ on $(\ell_\infty,\lVert \cdot \rVert_\infty)$ is semicompact which is not compact.

\section{$p$-$M$-Weakly and $p$-$L$-Weakly Compact Operators}

Recall that an operator $T\in B(X,Y)$ from a normed lattice $X$ into a normed space $Y$ is called \textit{M-weakly compact}, whenever $\lim\|Tx_n\|=0$ 
holds for every norm bounded disjoint sequence $x_n$ in $X$, and $T\in B(X,Y)$ from a normed space $X$ into a normed lattice $Y$ is called \textit{L-weakly compact}, 
whenever $\lim\|y_n\|=0$ holds for every disjoint sequence $y_n$ in $sol(T(B_X))$ (see for example, \cite[Def.3.6.9]{M}). Similarly we have:

\begin{definition} 
Let $T:(X,p,E)\to (Y,m,F)$ be a $p$-bounded and sequentially $p$-continuous operator between LNSs.
\begin{enumerate} 
\item[(1)] If $X$ is an LNVL and $m(Tx_n)\oc 0$ for every $p$-bounded disjoint sequence $x_n$ in $X$ then $T$ is said to be {\em $p$-$M$-weakly compact}.

\item[(2)] If $Y$ is an LNVL and $m(y_n)\oc 0$ for every disjoint sequence $y_n$ in $sol(T(A))$, where $A$ is a $p$-bounded subset of $X$, then $T$ is said to be {\em $p$-$L$-weakly compact}.
\end{enumerate}
\end{definition}

\begin{rem} \
\begin{enumerate}
\item[(1)] Let $(X,\lVert\cdot\rVert_X)$ be a normed lattice and $(Y,\lVert\cdot\rVert_Y)$ be a normed space. 
Assume $T\in B(X,Y)$ then $T:(X,\lVert\cdot\rVert_X,\mathbb{R})\to(Y,\lVert\cdot\rVert_Y,\mathbb{R})$ is $p$-$M$-weakly compact iff  $T:X\to Y$ is $M$-weakly compact.
\item[(2)] Let $(X,\lVert\cdot\rVert_X)$ be a normed space and $(Y,\lVert\cdot\rVert_Y)$ be a normed lattice. Assume $T\in B(X,Y)$ 
then $T:(X,\lVert\cdot\rVert_X,\mathbb{R})\to(Y,\lVert\cdot\rVert_Y,\mathbb{R})$ is $p$-$L$-weakly compact iff  $T:X\to Y$ is $L$-weakly compact.
\end{enumerate} 
\end{rem}

\noindent

In the sequel, the following fact will be used frequently.

\begin{rem}\label{disjointsequence}
If $x_n$ is a disjoint sequence in a vector lattice $X$ then $x_n\uoc 0$ $($see \cite[Cor.3.6]{GTX}$)$. If, in addition, $x_n$ is order bounded in $X$ then clearly $x_n\oc 0$.
\end{rem}

It is shown below that, in some cases, the collection of $p$-$M$ and $p$-$L$-weakly compact operators can be very large.

\begin{prop}\label{p-M-weakly-compact}
Assume $X$ to be a vector lattice and $(Y,\lVert\cdot\rVert_Y)$ a normed space. If $T:(X,\lvert\cdot\rvert,X)\to(Y,\lVert\cdot\rVert_Y,\mathbb{R})$ is $p$-bounded and sequentially $p$-continuous 
then $T$ is $p$-$M$-weakly compact.
\end{prop}

\begin{proof}
Let $x_n$ be a $p$-bounded disjoint sequence in $(X,\lvert\cdot\rvert,X)$. Then $x_n$ is order bounded in $X$ and, by Remark \ref{disjointsequence}, we get $x_n\oc 0$. 
That is, $x_n\pc 0$ in $(X,\lvert\cdot\rvert,X)$. Since $T$ is sequentially $p$-continuous then $Tx_n\xrightarrow{\lVert\cdot\rVert_Y}0$. 
Therefore, $T:(X,\lvert\cdot\rvert,X)\to(Y,\lVert\cdot\rVert_Y,\mathbb{R})$ is $p$-$M$-weakly compact.
\end{proof}

\begin{cor}
Let $(X,\lVert\cdot\rVert_X)$ be a normed lattice and $Y$ be a vector lattice. Let $Y_c^\sim$ denote the $\sigma$-order continuous dual of $Y$. 
If $0\leq T:(X,\lVert\cdot\rVert_X,\mathbb{R})\to(Y,\lvert\cdot\rvert,Y)$ is sequentially $p$-continuous and $p$-bounded 
then the operator $T^{\sim} :(Y_c^\sim,\lvert\cdot\rvert,Y_c^\sim)\to(X^*,\lVert\cdot\rVert_{X^*},\mathbb{R})$ defined by $T^{\sim} (f):=f\circ T$ is $p$-$M$-weakly compact.
\end{cor}

\begin{proof}
Theorem \ref{adjoint} implies that $T^\sim$ is $p$-continuous, and so it is $p$-bounded by Proposition \ref{p-cont.isp-bdd}. Thus, we get from Proposition \ref{p-M-weakly-compact}, that $T^\sim$ is $p$-$M$-weakly compact.
\end{proof}

\begin{prop}\label{p-L-weakly-compact}
Assume $(X,\lVert\cdot\rVert_X)$ to be a normed lattice and $Y$ a vector lattice. If $T:(X,\lVert\cdot\rVert_X,\mathbb{R})\to(Y,\lvert\cdot\rvert,Y)$ is $p$-bounded and sequentially $p$-continuous operator 
then $T$ is $p$-$L$-weakly compact.
\end{prop}

\begin{proof}
Let $A$ be a $p$-bounded set in  $(X,\lVert\cdot\rVert_X,\mathbb{R})$. Since $T$ is a $p$-bounded operator then $T(A)$ is $p$-bounded in $(Y,\lvert\cdot\rvert,Y)$, i.e. 
$T(A)$ is order bounded and hence $sol(T(A))$ is order bounded. Let $y_n$ be a disjoint sequence in $sol(T(A))$.  
Then, by Remark \ref{disjointsequence}, we have $y_n\oc 0$ in $Y$, i.e. $y_n\pc 0$ in $(Y,\lvert\cdot\rvert,Y)$. Thus, $T$ is $p$-$L$-weakly compact.
\end{proof}

\begin{cor}
Let $X$ be a vector lattice and $Y$ be an $AL$-space. Assume $0\leq T:(X,\lvert\cdot\rvert,X)\to(Y,\lVert\cdot\rVert_Y,\mathbb{R})$ to be sequentially $p$-continuous. 
Define $T^{\sim}:(Y^*,\lVert\cdot\rVert_{Y^*},\mathbb{R})\to(X^\sim,\lvert\cdot\rvert,X^\sim)$ by $T^{\sim}(f)=f\circ T$. Then $T^{\sim}$ is $p$-$L$-weakly compact.
\end{cor}

\begin{proof}
Theorem \ref{AL-adjoint} implies that $T^\sim$ is sequentially $p$-continuous and $p$-bounded, and so we get, by Proposition \ref{p-L-weakly-compact}, that $T^\sim$ is $p$-$L$-weakly compact.
\end{proof}

It is known that any order continuous operator is order bounded, but this fails for $\sigma$-order continuous operators; see \cite[Exer.10,p.289]{AB}. 
Therefore, we need the order boundedness condition in the following proposition.

\begin{prop}
If $T:X\to Y$ is an order bounded $\sigma$-order continuous operator between vector lattices then $T:(X,\lvert\cdot\rvert,X)\to(Y,\lvert\cdot\rvert,Y)$ is both $p$-$M$-weakly and $p$-$L$-weakly compact.
\end{prop}

\begin{proof}
Clearly, $T:(X,\lvert\cdot\rvert,X)\to(Y,\lvert\cdot\rvert,Y)$ is both sequentially $p$-continuous and $p$-bounded. 
	
First, we show that $T$ is $p$-$M$-weakly compact. Let $x_n$ be a $p$-bounded disjoint sequence of $X$. Then, by Remark \ref{disjointsequence}, we get $x_n\oc 0$ in $X$ and so $Tx_n\oc 0$ in $Y$. 
Therefore, $T$ is $p$-$M$-weakly compact.
	
Next, we show that $T$ is $p$-$L$-weakly compact. Let $A$ be a $p$-bounded set in $(X,\lvert\cdot\rvert,X)$ then $A$ is order bounded in $X$. 
Thus, $T(A)$ is order bounded and so $sol(T(A))$ is order bounded in $Y$. If $y_n$ is a disjoint sequence in $sol(T(A))$  
then again, by Remark \ref{disjointsequence}, $y_n\oc 0$ or $y_n\pc 0$ in $(Y,\lvert\cdot\rvert, Y)$. Therefore, $T$ is $p$-$L$-weakly compact.
\end{proof}

Next, we show that $p$-$M$-weakly and $p$-$L$-weakly compact operators satisfy the domination property.

\begin{prop}
Let $(X,p,E)$ and $(Y,m,F)$ be LNVLs and let $S,T:X\to Y$ be two linear operators such that $0\leq S\leq T$. 
\begin{enumerate}
\item[(i)] If $T$ is $p$-$M$-weakly compact then $S$ is $p$-$M$-weakly compact.
\item[(ii)] If $T$ is $p$-$L$-weakly compact then $S$ is $p$-$L$-weakly compact.
\end{enumerate}
\end{prop}

\begin{proof}
(i) Since $T$ is sequentially $p$-continuous and $p$-bounded then it is easily seen that $S$ is sequentially $p$-continuous and $p$-bounded.  
Let $x_n$ be a $p$-bounded disjoint sequence in $X$. Then $\lvert x_n\rvert$ is also $p$-bounded and disjoint. Since $T$ is $p$-$M$-weakly compact then $m(T\lvert x_n\rvert)\oc 0$ in $F$. 
Now, $0\leq S\lvert x_n\rvert\leq T\lvert x_n\rvert$ for all $n\in \mathbb{N}$ and since the lattice norm is monotone then we get $m(S\lvert x_n\rvert)\oc 0$ in $F$. 
Now, $\lvert Sx_n\rvert\leq S\lvert x_n\rvert$ for all $n\in \mathbb{N}$ and so $m(Sx_n)=m(\lvert Sx_n\rvert)\leq m(S\lvert x_n\rvert)\oc 0$ in $F$.  Thus, $S$ is $p$-$M$-weakly compact.
	
(ii) It is easy to see that  $S$ is sequentially $p$-continuous and $p$-bounded. Let $A$ be a $p$-bounded subset of $X$. 
Put $\lvert A\rvert=\{\lvert a\rvert: a\in A\}$. Clearly, $sol(S(A))\subseteq sol(S(\lvert A\rvert))$ and since $0\leq S\leq T$, we have $sol(S(\lvert A\rvert))\subseteq sol(T(\lvert A\rvert))$. 
Let $y_n$ be a disjoint sequence in $sol(S(A))$ then $y_n$ is in $sol(T(\lvert A\rvert))$ and, since $T$ is $p$-$L$-weakly compact then $m(S\lvert x_n\rvert)\oc 0$ in $F$. Therefore, $S$ is $p$-$L$-weakly compact.
\end{proof}

The following result is a variant of \cite[Thm.4.36]{AB}.

\begin{thm}\label{apprthm}
Let $(X,p,E )$ be a sequentially $p$-complete LNVL such that $(E,\left\|\cdot\right\|_E)$ is a Banach lattice, and let $(Y,m,F)$ be an LNS. 
Assume $T:(X,p,E)\to(Y,m,F)$ to be sequentially $p$-continuous, and let $A$ be a $p$-bounded solid subset of $X$.

If $m(Tx_n)\oc 0$ holds for each disjoint sequence $x_n$ in $A$ then, for each atom $a$ in $F$ and each $\varepsilon >0$, there exists $0\leq u\in I_A$ satisfying 
$$
  f_a\big(m(T(\lvert x\rvert-u)^+)\big)<\varepsilon
$$ 
for all $x\in A$, where $I_A$ denotes the ideal generated by $A$ in $X$.
\end{thm}

\begin{proof}
Suppose the claim is false. Then there is an atom $a_0\in F$ and $\varepsilon_0 >0$ such that, for each $u\geq 0$ in $I_A$, we have $f_{a_0}\big(m(T(\lvert x\rvert-u)^+)\big)\geq\varepsilon_0$ for some $x\in A$. 
In particular, there exists a sequence $x_n$ in $A$ such that
\begin{equation}\label{negation}
f_{a_0}\big(m(T(\lvert x_{n+1}\rvert-4^n\sum\limits_{i=1}^{n}\lvert x_i\rvert)^+)\big)\geq\varepsilon_0 \ \ \ (\forall n\in \mathbb{N}).
\end{equation}
Now, put $y=\sum\limits_{n=1}^{\infty}2^{-n}\lvert x_n\rvert$. Lemma \ref{mixednormisbanach} implies that $y\in X$.
Also let $w_n=(\lvert x_{n+1}\rvert-4^n\sum\limits_{i=1}^{n}\lvert x_i\rvert)^+$ and $v_n=(\lvert x_{n+1}\rvert-4^n\sum\limits_{i=1}^{n}\lvert x_i\rvert-2^{-n}y)^+$. By \cite[Lm.4.35]{AB}, 
the sequence $v_n$ is disjoint. Also since $A$ is solid and $0\leq v_n<\lvert x_{n+1}\rvert$ holds, we see that $v_n$ in $A$ and so, by the hypothesis, $m(Tx_n)\oc 0$.

On the other hand, $0\leq w_n-v_n\leq 2^{-n}y$ and so $p(w_n-v_n)\leq 2^{-n}p(y)$. Thus, $p(w_n-v_n)\oc 0$ in $F$. Since $T$ is sequentially $p$-continuous then $m\big(T(w_n-v_n)\big)\oc 0$  in $F$. Now, $m(Tw_n)\leq m(T(w_n-v_n))+m(Tv_n)$ implies that $m(Tw_n)\oc 0$ in $F$. In particular, $f_{a_0}\big(m(Tw_n)\big)\to 0$ as $n\to\infty$, which contradicts (\ref{negation}).
\end{proof}

\noindent

In \cite[Thm.5.60]{AB}, the approximation properties were provided for $M$-weakly and $L$-weakly compact operators. 
The following two propositions are similar to \cite[Thm.5.60]{AB} in the case of $p$-$M$-weakly and $p$-$L$-weakly compact operators.

\begin{prop}\label{aprrxpMweakly}
Let $(X,p,E)$ be a sequentially $p$-complete LNVL with a Banach lattice $(E,\left\|\cdot\right\|_E)$, $(Y,m, F)$ be an LNS, $T:(X,p,E)\to(Y,m,F)$ be $p$-$M$-weakly compact, 
and $A$ be a $p$-bounded solid subset of $X$. Then, for each atom $a$ in $F$ and each $\varepsilon >0$, there exists some $u\in X_+$ such that 
$$
  f_a\big(m(T(\lvert x\rvert-u)^+)\big)<\varepsilon
$$ 
holds for all $x\in A$.
\end{prop}

\begin{proof}
Let $A$ be a $p$-bounded solid subset of $X$. Since $T$ is $p$-$M$-weakly compact then $m(Tx_n)\oc 0$ for every disjoint sequence in $A$. 
By Theorem \ref{apprthm}, for any atom $a\in F$ and any $\varepsilon >0$, there exists $u\in X_+$ such that $f_a\big(m(T(\lvert x\rvert-u)^+)\big)<\varepsilon$ for all $x\in A$.
\end{proof}

\begin{prop}
Let $(X,p,E)$ be an LNS and $(Y,m,F)$ be a sequentially $p$-complete LNVL with a Banach lattice $F$. 
Assume $T:(X,p,E)\to(Y,m,F)$ to be $p$-$L$-weakly compact and $A$ to be $p$-bounded in $X$. 
Then, for each atom $a$ in $F$ and each $\varepsilon>0$, there exists some $u\in Y_+$ in the ideal generated by $T(X)$ satisfying  
$$
  f_a\big(m(\lvert Tx\rvert-u)^+)\big)<\varepsilon
$$ 
for all $x\in A$.
\end{prop}

\begin{proof}
Let $A$ be a $p$-bounded subset of $X$. Since $T$ is $p$-$L$-weakly compact, $m(y_n)\oc 0$ for any disjoint sequence $y_n$ in $sol(T(A))$. 
Consider the identity operator $I$ on $(Y,m,F)$. By Theorem \ref{apprthm}, for any atom $a\in F$ and each $\varepsilon >0$, 
there exists  $u\in Y_+$ in the ideal generated by $sol(T(A))$ (and so in the ideal generated by $T(X)$)  such that  
$$
  f_a \big(m(|y|-u)^+)\big)<\varepsilon
$$ 
for all $y\in sol(T(A))$. In particular, 
$$
  f_a\big(m(|Tx|-u)^+)\big)<\varepsilon
$$ 
for all $x\in A$.
\end{proof}

The next two results provide relations between $p$-$M$-weakly and $p$-$L$-weakly compact operators, which are known for $M$-weakly and $L$-weakly compact operators; e.g. \cite[Thm.5.67 and Exer.4(a),p:337]{AB}

\begin{thm}
Let $(X,p,E)$ be a sequentially $p$-complete LNVL with a norming Banach lattice $(E,\lVert\cdot\rVert_E)$, $(Y,m,F)$ be an $op$-continuous LNVL with an atomic norming lattice $F$ and $T\in L^\sim(X,Y)$. 
If $T:(X,p,E)\to(Y,m,F)$ is $p$-$M$-weakly compact then $T$ is $p$-$L$-weakly compact.
\end{thm}

\begin{proof}
Let $A$ be a $p$-bounded subset of $X$ and let $y_n$ be a disjoint sequence in $sol(T(A))$. 
Then there is a sequence $x_n$ in $A$ such that $\lvert y_n\rvert\leq\lvert Tx_n\rvert$ for all $n\in \mathbb{N}$. Let $a\in F$ be an atom. 
Given $\varepsilon >0$ then, by Proposition \ref{aprrxpMweakly}, there is $u\in X_+$ such that 
$$
  f_a\big(m(T(\lvert x\rvert-u)^+)\big)<\varepsilon
$$ 
holds for all $x\in sol(A)$. In particular, for all $n\in \mathbb{N}$, we have 
$$
  f_a\big(m(T( x_n^+-u)^+)\big)<\varepsilon \ \ \ \text{and} \ \ \ f_a\big(m(T( x_n^--u)^+)\big)<\varepsilon
$$ 
Thus, for each $n\in \mathbb{N}$,
\begin{eqnarray*}
\lvert y_n\rvert &\leq&\lvert Tx_n\rvert\leq\lvert Tx^+_n\rvert+\lvert Tx^-_n\rvert \\ 
&=& \lvert T(x^+_n-u)^++ T(x^+_n\wedge u )\rvert +\rvert T(x^-_n-u)^++ T(x^-_n\wedge u)\rvert \\ 
&\leq&\lvert T(x^+_n-u)^+\rvert+\lvert T(x^+_n\wedge u)\rvert+\lvert T(x^-_n-u)^+\rvert+\lvert T(x^-_n\wedge u)\rvert \\ 
&\leq&\lvert T(x^+_n-u)^+\rvert+\lvert T(x^-_n-u)^+\rvert+\lvert T\rvert(x^+_n\wedge u)+\lvert T\rvert(x^-_n\wedge u) \\ 
&\leq&\lvert T(x^+_n-u)^+\rvert+\lvert T(x^-_n-u)^+\rvert+2\lvert T\rvert u.
\end{eqnarray*}
By Riesz decomposition property, for all $n\in \mathbb{N}$, there exist $u_n,v_n\geq 0$ such that $y_n=u_n+v_n$ and $0\leq u_n\leq\lvert T(x^+_n-u)^+\rvert +\lvert T(x^-_n-u)^+\rvert$, $0\leq v_n\leq 2\lvert T\rvert u$. 
Since $y_n$ is disjoint sequence and $v_n\leq\lvert y_n\rvert$ for all $n\in \mathbb{N}$ then the sequence $v_n$ is disjoint. 
Moreover, it is order bounded. Hence, $v_n\oc o$. Since $(Y,m,F)$ is $op$-continuous then $m(v_n)\oc 0$. In particular, $f_a\big(m(v_n)\big)\to 0$ as $n\to\infty$. So, for given $\varepsilon>0$, there is $n_0\in \mathbb{N}$ such that $f_a (m(v_n))<\varepsilon$ for all $n\geq n_0$.  Thus, for any $n\geq n_0$, we have
\begin{eqnarray*}
f_a\big(m(y_n)\big)&\leq& f_a\big(m(u_n)\big)+f_a\big(m(v_n)\big) \\& \leq&  f_a\big(m(T(x^+_n-u)^+)\big)+f_a\big(m(T(x^-_n-u)^+)\big)+\varepsilon\leq 3\varepsilon. 
\end{eqnarray*}
Hence, $f_a\big(m(y_n)\big)\to 0$ as $n\to\infty$. Since $T$ is $p$-bounded then $m(y_n)$ is order bounded. The atomicity of $F$ implies $m(y_n)\oc 0$ in $F$. Therefore, $T$  is $p$-$L$-weakly compact.
\end{proof}

\begin{prop}
Let $(X,p,E)$ and $(Y,m,F)$ be LNVLs. If $ T:(X,p,E)\to(Y,m,F)$ is a $p$-$L$-weakly compact lattice homomorphism then $T$ is $p$-$M$-weakly compact.
\end{prop}

\begin{proof}
Let $x_n$ be a $p$-bounded disjoint sequence in $X$. Since $T$ is lattice homomorphism then we have that $Tx_n$ is disjoint in $Y$. 
Clearly $Tx_n\in sol\big(\{Tx_n:n\in \mathbb{N}\}\big)$. Since $T$ is a $p$-$L$-weakly compact lattice homomorphism 
then $m\big(T(x_n)\big)\oc 0$ in $F$. Therefore, $T$ is $p$-$M$-weakly compact.
\end{proof}
We end up this section by an investigation of the relation between $p$-$M$-weakly (respectively, $p$-$L$-weakly) compact operators and $M$-weakly (respectively, $L$-weakly) compact operators acting in mixed-normed spaces. 
\begin{prop}
Given an LNVL $(X,p,E)$ with $(E,\lVert\cdot\rVert_E)$, which is an $AM$-space with a strong unit. Let an LNS $(Y,m,F)$ be such that $(F,\lVert\cdot\rVert_F)$ is a $\sigma$-order continuous normed lattice. 
If $T:(X,p,E)\to(Y,m,F)$ is $p$-$M$-weakly compact then $T:(X,p\text{-}\lVert\cdot\rVert_E)\to(Y,m\text{-}\lVert\cdot\rVert_F)$ is $M$-weakly compact.
\end{prop}

\begin{proof}
By Proposition \ref{seqpcontimpliesnormcont}, it follows that $T:(X,p\text{-}\lVert\cdot\rVert_E)\to(Y,m\text{-}\lVert\cdot\rVert_F)$ is norm continuous. Let $x_n$ be a norm bounded disjoint sequence in $(X,p\text{-}\lVert\cdot\rVert_E)$. 
Then $p\text{-}\lVert x_n\rVert_E\leq M<\infty$ or $\lVert p(x_n)\rVert_E\leq M<\infty$ for all $n\in \mathbb{N}$. Since $(E,\lVert\cdot\rVert_E)$ is an $AM$-space with a strong unit then there is $e\in E_+$ 
such that $p(x_n)\leq e$  for all $n\in \mathbb{N}$. Thus, $x_n$ is a $p$-bounded disjoint sequence in $(X,p,E)$. Since $T:(X,p,E)\to(Y,m,F)$ is $p$-$M$-weakly compact then $m(Tx_n)\oc 0$ in $F$. 
It follows from the $\sigma$-order continuity of $(F,\lVert\cdot\rVert_F)$, that $\lVert m(Tx_n)\rVert_F\to 0$ or $\lim\limits_{n\to\infty}m\text{-}\lVert Tx_n\rVert_F=0$. 
Therefore, $T:(X,p\text{-}\lVert\cdot\rVert_E)\to(Y,m\text{-}\lVert\cdot\rVert_F)$ is $M$-weakly compact.
\end{proof}

\begin{prop}
Suppose $(X,p,E)$ to be an LNVL with a $\sigma$-order continuous normed lattice $(E,\lVert\cdot\rVert_E)$ and $(Y,m,F)$ to be an LNS with an atomic normed lattice $(F,\lVert\cdot\rVert_F)$. 
Assume further that$:$
\begin{itemize}
\item[(i)] $T:(X,p,E)\to(Y,m,F)$ is $p$-bounded; 
\item[(ii)] $T:(X,p\text{-}\lVert\cdot\rVert_E)\to(Y,m\text{-}\lVert\cdot\rVert_F)$ is $M$-weakly compact. 
\end{itemize}
Then $T:(X,p,E)\to(Y,m,F)$ is $p$-$M$-weakly compact.
\end{prop}

\begin{proof}
The assumptions, together with Theorem \ref{normcontimpliespcont}, imply that $T:(X,p,E)\to(Y,m,F)$ is sequentially $p$-continuous.

Assume $x_n$ to be a $p$-bounded disjoint sequence in $(X,p,E)$. Then $x_n$ is disjoint and norm bounded in $(E,p\text{-}\lVert\cdot\rVert_E)$. 
Since $T:(X,p\text{-}\lVert\cdot\rVert_E)\to(Y,m\text{-}\lVert\cdot\rVert_F)$ is $M$-weakly compact then $\lim\limits_{n\to\infty}m\text{-}\lVert Tx_n\rVert_F=0$ or $\lim\limits_{n\to\infty}\lVert m(Tx_n)\rVert_F$ $=0$. Since $x_n$ is $p$-bounded and $T:(X,p,E)\to(Y,m,F)$ is $p$-bounded then $m(Tx_n)$ is order bounded in $F$. Let $a\in F$ be an atom then  
$$
  \big\lvert f_a\big(m(Tx_n)\big)\big\rvert\leq\lVert f_a\rVert\lVert m(Tx_n)\rVert_F\to 0 \ \ \ \text{as}  \ n\to\infty.
$$ 
Since $F$ is atomic then $m(Tx_n)\oc 0$. Therefore, $T:(X,p,E)\to(Y,m,F)$ is $p$-$M$-weakly compact.
\end{proof}

\begin{prop}
Assume $(X,p,E)$ to be an LNS with an $AM$-space $(E,\lVert\cdot\rVert_E)$ possessing a strong unit, and $(Y,m,F)$ to be an LNVL with a $\sigma$-order continuous normed lattice $(F,\lVert\cdot\rVert_F)$.
If $T:(X,p,E)\to(Y,m,F)$ is $p$-$L$-weakly compact then $T:(X,p\text{-}\lVert\cdot\rVert_E)\to(Y,m\text{-}\lVert\cdot\rVert_F)$ is $L$-weakly compact.
\end{prop}

\begin{proof}
Proposition \ref{seqpcontimpliesnormcont} implies that $T:(X,p\text{-}\lVert\cdot\rVert_E)\to(Y,m\text{-}\lVert\cdot\rVert_F)$ is norm continuous. 
Let $B_X$ be the closed unit ball of  $(X,p\text{-}\lVert\cdot\rVert_E)$. Then $p\text{-}\lVert x\rVert_E\leq 1$ or $\lVert p(x)\rVert_E\leq 1$ for all $x\in B_X$. 
Since $(E,\lVert\cdot\rVert_E)$ is an $AM$-space with a strong unit then there is an element $e\in E_+$ such that $p(x)\leq e$ for each $x\in B_X$. So $B_X$ is $p$-bounded. 
Let $y_n$ be a disjoint sequence in $sol(T(B_X))$. Since $T:(X,p,E)\to(Y,m,F)$ is $p$-$L$-weakly compact then $m(y_n)\oc 0$ in $F$. Since  $(F,\lVert\cdot\rVert_F)$ is $\sigma$-order continuous normed lattice 
then $\lVert m(y_n)\rVert_F\to 0$ or $\lim\limits_{n\to\infty}m\text{-}\lVert y_n\rVert_F = 0$. So $T:(X,p\text{-}\lVert\cdot\rVert_E)\to(Y,m\text{-}\lVert\cdot\rVert_F)$ is $L$-weakly compact.
\end{proof} 

\begin{prop}
Let $(X,p,E)$ be an LNS with a $\sigma$-order continuous normed lattice, $(Y,m,F)$ be an LNVL with an atomic normed lattice $(F,\lVert\cdot\rVert_F)$. Assume that$:$
\begin{enumerate}
\item[(i)] $T:(X,p,E)\to(Y,m,F)$ is $p$-bounded, and
\item[(ii)] $T:(X,p\text{-}\lVert\cdot\rVert_E)\to(Y,m\text{-}\lVert\cdot\rVert_F)$ is $L$-weakly compact.
\end{enumerate}
Then $T:(X,p,E)\to(Y,m,F)$ is $p$-$L$-weakly compact.
\end{prop}

\begin{proof}
Theorem \ref{normcontimpliespcont} implies that $T:(X,p,E)\to(Y,m,F)$ is sequentially $p$-continuous. Let $A$ be a $p$-bounded set. 
Then there is $e\in E_+$ such that $p(a)\leq e$ for all $a\in A$. Hence, $\lVert p(a)\rVert_E\leq\lVert e\rVert_E$ for all $a\in A$ or $p\text{-}\lVert a\rVert_E\leq\lVert e\rVert_E$ for each $a\in A$. 
Thus, $A$ is norm bounded in $(X,p\text{-}\lVert\cdot\rVert_E)$. Let $y_n$ be a disjoint sequence in $sol(T(A))$. Since $T:(X,p\text{-}\lVert\cdot\rVert_E)\to(Y,m\text{-}\lVert\cdot\rVert_F)$ is $L$-weakly compact 
then $\lim\limits_{n\to\infty}m\text{-}\lVert y_n\rVert_F=0$ or $\lim\limits_{n\to\infty}\lVert m(y_n)\rVert_F=0$.\\
Since $T:(X,p,E)\to(Y,m,F)$ is $p$-bounded and $A$ is $p$-bounded then $T(A)$ is $p$-bounded in $Y$ and so $sol(T(A))$ is $p$-bounded in $Y$. 
Hence, $y_n$ is a $p$-bounded sequence in $(Y,m,F)$;  i.e. $m(y_n)$ is order bounded in $F$. 
Let $a\in F$ be an atom and consider its biorthogonal functional $f_a$. Then 
$$
  \lvert f_a\big(m(y_n)\big)\rvert\leq\lVert f_a\rVert\lVert m(y_n)\rVert_F\to 0\ \ \text{as}\ n\to\infty.
$$  
So, for any atom $a\in F$, $\lim\limits_{n\to\infty}f_a\big(m(y_n)\big)=0$ and, since $m(y_n)$ is order bounded in an atomic vector lattice $F$, $m(y_n)\oc 0$ in $F$. 
Thus, $T$ is $p$-$L$-weakly compact.
\end{proof}

\section{$up$-Continuous and $up$-Compact Operators}

Using the $up$-convergence in LNVLs, we introduce the following notions. 

\begin{definition}
Let $X$, $Y$ be two LNVLs and $T\in L(X,Y)$. Then$:$
\begin{enumerate}
\item[(1)] $T$ is called {\em $up$-continuous} if $x_\alpha\upc 0$ in $X$ implies $Tx_\alpha\upc 0$ in $Y$, 
if the condition holds for sequences then $T$ is called {\em sequentially $up$-continuous}$;$
\item[(2)] $T$ is called {\em $up$-compact} if for any $p$-bounded net $x_\alpha$ in $X$ 
there is a subnet $x_{\alpha_\beta}$ such that $Tx_{\alpha_\beta}\upc y$ in $Y$ for some $y\in Y$$;$ 
\item[(3)] $T$ is called {\em sequentially-$up$-compact} if for any $p$-bounded sequence $x_n$ in $X$ 
there is a subsequence $x_{n_k}$ such that $Tx_{n_k}\upc y$ in $Y$ for some $y\in Y$. 
\end{enumerate}
\end{definition}

\begin{rem} \ \
\begin{enumerate}
\item[(i)] The notion of $up$-continuous operators is motivated by two recent notions, namely$:$ $\sigma$-unbounded order continuous $($$\sigma uo$-continuous$)$ mappings between vector lattices 
$($see \cite[p.23]{EM}$)$, and $un$-continuous functionals on Banach lattices $($see \cite[p.17]{KMT}$)$.
\item[(ii)] If $T$ is $($sequentially$)$ $p$-continuous operator then $T$ is $($sequentially$)$ $up$-continuous.
\item[(iii)] If $T$ is $($sequentially$)$ $p$-compact operator then $T$ is $($sequentially$)$ $up$-compact.
\item[(iv)] Let $(X,\lVert\cdot\rVert_X)$ be a normed space and $(Y,\lVert\cdot\rVert_Y)$ be a normed lattice. An operator $T\in B(X,Y)$ is called {\em $($sequentially$)$ $un$-compact} if 
for every norm bounded net $x_\alpha$ $($respectively, every norm bounded sequence $x_n$$)$, its image has a subnet $($respectively, subsequence$)$, which is $un$-convergent; see \cite[Sec.9,p.28]{KMT}. 
Therefore, $T\in B(X,Y)$ is $($sequentially$)$ $un$-compact iff $T:(X,\lVert\cdot\rVert_X,\mathbb{R})\to(Y,\lVert\cdot\rVert_Y,\mathbb{R})$ 
is $($sequentially$)$ $up$-compact.
\end{enumerate}
\end{rem}

\begin{prop}
Let $(X,E)$, $(Y,F)$ be two LNVLs and $T\in L(X,Y)$. If $T$ is $up$-compact and $p$-semicompact operator then $T$ is $p$-compact.
\end{prop}

\begin{proof}
Let $x_\alpha$ be a $p$-bounded net in $X$. Then $Tx_\alpha$ is $p$-almost order bounded net in $Y$, as $T$ is $p$-semicompact operator.  
Moreover, since $T$ is $up$-compact then there is a subnet $x_{\alpha_\beta}$ such that $Tx_{\alpha_\beta}\upc y$ for some $y\in Y$. 
It follows by \cite[Prop.9]{AEEM}, that $Tx_{\alpha_\beta}\pc y$. 
Therefore, $T$ is $p$-compact.
\end{proof}

Similar to Proposition \ref{leftandrightmultiplication}, for any $S,T \in L(X)$, where $X$ is an LNVL the following holds:
\begin{enumerate}
\item[(i)] If $S$ is $p$-bounded and  $T$ is $up$-compact then $T\circ S$ is $up$-compact.
\item[(ii)] If $S$ is $up$-continuous and $T$ is $up$-compact then $S\circ T$ is $up$-compact.
\end{enumerate}

Now we investigate composition of a sequentially $up$-compact operator 
with a dominated lattice homomorphism.
\begin{thm}\label{sequentially $up$-compactness}
Let $(X,p,E)$ be an LNVL, $(Y,m,F)$ an LNVL with an order continuous Banach lattice $(F,\lVert\cdot\rVert_F)$, and $(Z,q,G)$ an LNVL with a Banach lattice $(G,\lVert\cdot\rVert_G)$. 
If $T\in L(X,Y)$ is a sequentially $up$-compact operator and $S\in L(Y,Z)$ is a dominated surjective lattice homomorphism then $S\circ T$ is sequentially $up$-compact.
\end{thm}

\begin{proof}
Let $x_n$ be a $p$-bounded sequence in $X$. Since $T$ is sequentially $up$-compact then there is a subsequence $x_{n_k}$ such that $Tx_{n_k}\upc y$ in $Y$ for some $y\in Y$. 
Let $u\in Z_+$. Since $S$ is surjective lattice homomorphism, we have some $v\in Y_+$ such that $Sv=u$. Since $Tx_{n_k}\upc y$ then $m(\lvert Tx_{n_k}-y\rvert\wedge v)\oc 0$ in $F$. 
Clearly, $F$ is order complete and so, by \cite[Prop.1.5]{YG}, there are $f_k\downarrow 0$ and $k_0\in \mathbb{N}$ such that
\begin{eqnarray}\label{dominatingsequence}
m(\lvert Tx_{n_k}-y\rvert\wedge v)\leq f_k \ \ \ \ \ \ (k\geq k_0).
\end{eqnarray}
Note also  $\lVert f_k\rVert_F\downarrow 0$ in $F$, as $(F,\lVert\cdot\rVert_F)$ is an order continuous Banach lattice. Since $S$ is dominated then there is a positive operator $R:F\to G$ such that 
$$
  q\big(S(\lvert Tx_{n_k}-y\rvert\wedge v)\big)\leq R\big(m(\lvert Tx_{n_k}-y\rvert\wedge v)\big).
$$ 
Taking into account that $S$ is a lattice homomorphism and $Sv=u$, we get, by (\ref{dominatingsequence}), that
\begin{eqnarray}\label{Rdominatingsequence}
q(\lvert S\circ Tx_{n_{k}}-Sy\rvert\wedge u)\leq Rf_k\ \ \ \ \ \ (k\geq k_0).
\end{eqnarray}
Since $R$ is positive then by \cite[Thm.4.3]{AB} it is norm continuous. Hence, $\lVert Rf_k\rVert_G\downarrow 0$.  
Also, by \cite[Thm.VII.2.1]{V}, there is a subsequence $f_{k_j}$ of $(f_k)_{k\geq k_0}$ such that $Rf_{k_j}\oc 0$ in $G$, and so $Rf_{k_j}\downarrow 0$ in $G$. So (\ref{Rdominatingsequence}) becomes 
$$
  q(\lvert S\circ Tx_{n_{k_j}}-Sy\rvert\wedge u)\leq Rf_{k_j}\ \ \ \ \ \ (j\in\mathbb{N}).
$$
Since $u\in Z_+$ is arbitrary, $S\circ T(x_{n_{k_j}})\upc Sy$. Therefore, $S\circ T$ is sequentially $up$-compact.
\end{proof}

\begin{rem}
In connection with the proof of Theorem \ref{sequentially $up$-compactness} it should be mentioned that, if the operator $T$ is $up$-compact  
and $S$ is a surjective lattice homomorphism with an order continuous dominant then it can be easily seen that $S\circ T$ is $up$-compact.
\end{rem}

Recall that, for an LNVL $(X,p,E)$, a sublattice $Y$ of $X$ is called {\em $up$-regular} if, for any net $y_\alpha$ in $Y$, 
the convergence $y_\alpha\upc 0$ in $Y$ implies $y_\alpha\upc 0$ in $X$; see \cite[Def.10 and Sec.3.4]{AEEM}.

\begin{cor}\label{rangeisup-regular}
Let $(X,p,E)$ be an LNVL, $(Y,m,F)$ an LNVL with an order continuous Banach lattice $(F,\lVert\cdot\rVert_F)$, 
and $(Z,q,G)$ an LNVL with a Banach lattice $(G,\lVert\cdot\rVert_G)$. If  $T\in L(X,Y)$ is a sequentially $up$-compact operator, 
$S\in L(Y,Z)$ is a dominated lattice homomorphism, and $S(Y)$ is $up$-regular in $Z$ then $S\circ T$ is sequentially $up$-compact.
\end{cor}

\begin{proof}
Since $S$ is a lattice homomorphism then $S(Y)$ is a vector sublattice of $Z$. So $(S(Y),q,G)$ is an LNVL. Thus, by Theorem \ref{sequentially $up$-compactness}, 
we have $S\circ T:(X,p,E)\to(S(Y),q,G)$ is sequentially $up$-compact.

Next, we show that $S\circ T:(X,p,E)\to(Z,q,G)$ is sequentially $up$-compact. Let $x_n$ be a $p$-bounded sequence in $X$. 
Then there is a subsequence $x_{n_k}$ such that $S\circ T(x_{n_k})\upc z$ in $S(Y)$ for some $z\in S(Y)$. Since $S(Y)$ is $up$-regular in $Z$, 
we have  $S\circ T(x_{n_k})\upc z$ in $Z$. Therefore, $S\circ T:X\to Z$ is sequentially $up$-compact.
\end{proof}

The next result is similar to \cite[Prop.9.4.]{KMT}.

\begin{cor}\label{idealisup-regular}
Let $(X,p,E)$ be an LNVL, $(Y,m,F)$ an LNVL with an order continuous Banach lattice $(F,\lVert\cdot\rVert_F)$, and $(Z,q,G)$ an LNVL with a Banach lattice $(G,\lVert\cdot\rVert_G)$. 
If $T\in L(X,Y)$ is a sequentially $up$-compact operator, $S\in L(Y,Z)$ is a dominated lattice homomorphism, and $I_{S(Y)}$ $($the ideal generated by $S(Y)$$)$ is $up$-regular in $Z$ 
then $S\circ T$ is sequentially $up$-compact.
\end{cor}

\begin{proof}
Let $x_n$ be a $p$-bounded sequence in $X$. Since $T$ sequentially $up$-compact, there exist a subsequence $x_{n_k}$ and $y_0\in Y$ such that $Tx_{n_k}\upc y_0$ in $Y$. 
Let $0\leq u\in I_{S(Y)}$. Then there is $y\in Y_+$ such that $0\leq u\leq Sy$. Therefore, we have for a dominant $R$:
\begin{eqnarray*}
&&q\big(S(\lvert Tx_{n_k}-y_0\rvert\wedge y)\big)\leq R\big(m(\lvert Tx_{n_k}-y_0\rvert\wedge y)\big)
\end{eqnarray*}
and so 
$$ 
  q\big((\lvert STx_{n_k}-Sy_0\rvert\wedge Sy)\big)\leq R\big(m(\lvert Tx_{n_k}-y_0\rvert\wedge y)\big).
$$
It follows from $0\leq u\leq Sy$, that
$$
  q\big((\lvert STx_{n_k}-Sy_0\rvert\wedge u)\big)\leq R\big(m(\lvert Tx_{n_k}-y_0\rvert\wedge u)\big).
$$

Now, the argument given in the proof of Theorem \ref{sequentially $up$-compactness} can be repeated here as well. 
Thus, we have that $S\circ T:(X,p,E)\to(I_{S(Y)},q,G)$ is sequentially $up$-compact. 
Since $I_{S(Y)}$ is $up$-regular in $Z$ then it can be easily seen that $S\circ T:X\to Z$ is sequentially $up$-compact.
\end{proof}

We conclude this section by a result which might be compared with Proposition 9.9 in \cite{KMT}.

\begin{prop}
Let $(X,p,E)$ be an LNS and let $(Y,\lVert\cdot\rVert_Y)$ be a $\sigma$-order continuous normed lattice. 
If $T:(X,p,E)\to(Y,\lvert\cdot\rvert,Y)$ is sequentially $up$-compact and $p$-bounded then $T:(X,p,E)\to(Y,\lVert\cdot\rVert_Y)$ is $GAM$-compact.
\end{prop}

\begin{proof}
Let $x_n$ be a $p$-bounded sequence in $X$. Since $T$ is $up$-compact, there exist a subsequence $x_{n_k}$ and some $y\in Y$ such that  
$Tx_{n_k}\upc y$ in $(Y,\lvert\cdot\rvert,Y)$ and, by the  $\sigma$-order continuity of $(Y,\lVert\cdot\rVert_Y)$, we have $Tx_{n_k}\unc y$ in $Y$. 
Moreover, since $T$ is $p$-bounded then $Tx_n$ is $p$-bounded $(Y,\lvert\cdot\rvert,Y)$ or order bounded in $Y$, and so we get $Tx_{n_k}\xrightarrow{\lVert\cdot\rVert_Y} y$. Therefore, $T$ is $GAM$-compact.
\end{proof}

\end{document}